\documentclass[10pt,reqno]{article}

\usepackage{graphicx}
\usepackage{amssymb}
\usepackage{epsfig}
\usepackage{amsmath}
\usepackage{amsthm}
\usepackage{color}
\definecolor{r}{rgb}{0.9,0.3,0.1}
\definecolor{b}{rgb}{0.1,0.3,0.9}

\newtheorem{theo}{Theorem}[section]

\newtheorem{lemm}[theo]{Lemma}
\newtheorem{rem}[theo]{Remark}

\newcommand{\al}{\alpha}
\newcommand{\be}{\beta}
\newcommand{\ga}{\gamma}
\newcommand{\Ga}{\Gamma}

\newcommand{\om}{\omega}
\newcommand{\Om}{\Omega}

\newcommand{\si}{\sigma}

\newcommand{\ep}{\epsilon }
\newcommand{\te}{\theta}
\newcommand{\De}{\Delta}
\newcommand{\de}{\delta}

\newcommand{\pa}{\partial}

\newcommand{\ri}{\rightarrow}

\newcommand{\na}{\nabla}

\newcommand{\Rb}{{\mathbb R}^n }

\newcommand{\mysection}[1]{\section{#1}
\setcounter{equation}{0}}

\begin{document}
\setlength{\baselineskip}{16pt}

\title%[Boundary SPDE]
{On Initial-Boundary Value Problem of\\ Stochastic Heat Equation
in a Lipschitz Cylinder}
\author{TongKeun Chang$\,^1$, Kijung Lee$\,^2$, Minsuk Yang$\,^{1}$
\\[0.5cm]
\small{$\;^{1}\,$Department of Mathematics, Yonsei University}\\
\small{Seoul, 136-701, Korea}\\
\small{$\;^{2}\,$Department of Mathematics, Ajou University}\\
\small{Suwon, 443-749, Korea}\\}
\date{}

%\email{chang7357@yonsei.ac.kr}
%\email{kijung@ajou.ac.kr}
%\email{kusnim@gmail.com}

%\thanks{ }

\maketitle

\begin{abstract}
We consider the initial boundary value problem of non-homogeneous
stochastic heat equation. The derivative of the solution with
respect to time receives heavy random perturbation. The space
boundary is Lipschitz and we impose non-zero cylinder condition.
We prove a regularity result after finding suitable spaces for the
solution and the pre-assigned datum  in the problem. The tools
from potential theory, harmonic analysis and probability are used.
Some Lemmas are as important as the main Theorem.

\vspace*{.125in}


\vspace*{.125in}

\end{abstract}

\mysection{Introduction}\label{introduction}
\setcounter{equation}{0}

We study the following initial boundary value problem:
\begin{eqnarray}\label{main equation}
\left\{\begin{array}{l} d u(t,x)= (\De u(t,x) + f(t,x))dt + g(t,x)
d w_t, \quad  (t,x)\in (0,T) \times D,\vspace{0.3cm}\\
\;\;u(t,x) = b(t,x),\hspace{4.6cm} (t,x)\in(0,T) \times \pa D,\vspace{0.3cm}\\
\;\;u(0,x) = u_0,\hspace{5.7cm} x\in D,
\end{array}
\right.
\end{eqnarray}
where $D$ is a bounded Lipschitz domain in $\Rb$ and
$\{w_t(\omega):t\ge 0,\omega\in\Omega\}$ is a one-dimensional
Brownian motion with a probability space $\Omega$. Any solution of
(\ref{main equation}) depends not only $(t,x)$, but also $\omega$.
We investigate the regularity of the solution of (\ref{main
equation}) in $(t,x)$ for each $\omega$.

If $g\equiv 0$, the problem is deterministic and the theory has
been well-developed. For instance, \cite{FR} considered the
problem when $D$ is a bounded $C^1$-domain and  \cite{B} and
\cite{B2} studied the problem when $D$ is a  bounded Lipschitz
domain. Later, \cite{JM} developed a theory using anisotropic
Besov spaces. However in our paper, as we let $g\not\equiv 0$, we
deal with a stochastic heat equation. This job is nontrivial.
Viewing the heat equation in (\ref{main equation}) as
$u_t(t,x)=\De u(t,x)+f(t,x)+\dot{w}_t g(t,x)$, we notice that our
equation includes an internal source/sink with the white noise
coefficient. The (probabilistic) variance of the random noise
$\dot{w}_t$, $t\in (0,T)$ is not bounded. Moreover $\dot{w}_{t_1}$
and $\dot{w}_{t_2}$ are independent as long as $t_1\ne t_2$. Thus,
 \emph{we do not expect good regularity in time direction since the
solution keeps receiving the white noises  along the time
variable.} An $L_p$-theory of the Cauchy problem
($D=\mathbb{R}^n$) was established in \cite{K1}  and since then
the initial boundary value problem with zero boundary condition is
studied by many authors (see, for instance, \cite{KL1},
\cite{KL2}, \cite{KK}, \cite{Kim03}, \cite{Lo} and references
therein). In this paper we allows the space domain to be Lipschitz
and the boundary condition can be non-zero. Moreover, when we do
not require the high regularity in $x$, we consider the joint
regularity in $(t,x)$ using anisotropic Besov spaces. The usage of
anisotropic Besov spaces is natural with the deterministic heat
equation.

Having said that, let us find a formal solution of (\ref{main
equation});  this will be a unique solution in an appropriate space.
Firstly, extend  $u_0$ on $\Rb$, $f$ and $g$ on $(0,T)\times\Rb$
(see Section \ref{Lemmas} for the mathematical details on these
extensions). Let $v$ be a solution of the Cauchy problem, i.e.
$D=\mathbb{R}^n$, consisting of \eqref{main equation} with the
extended $u_0$ as the initial condition. Let $\hat h$ denote the
Fourier transform of a function $h$ in $\Rb$. Taking Fourier
transform in space on  the equation, we have a stochastic
differential equation for each frequency $\xi\in\mathbb{R}^n$,
\begin{eqnarray}
d\hat{v}(t,\xi)=(-|\xi|^{2}\hat{v}(t,\xi)+\hat{f}(t,\xi)\;)dt+
\hat{g}(t,\xi)dw_t.\nonumber
\end{eqnarray}
Putting the terms with $\hat{v}$ together in the left hand side,
we get
\begin{eqnarray}
d\left(\hat{v}(t,\xi)\;e^{|\xi|^{2}t}\right)=
e^{|\xi|^{2}t}\hat{f}(t,\xi)dt+
e^{|\xi|^{2}t}\hat{g}(t,\xi)dw_t\nonumber
\end{eqnarray}
and hence
\begin{eqnarray}
\hat{v}(t,\xi)=e^{-|\xi|^{2}t}\hat{u}_0(\xi)+
\int^t_0e^{-|\xi|^{2}(t-s)}\hat{f}(s,\xi)ds+
\int^t_0e^{-|\xi|^{2}(t-s)}\hat{g}(s,\xi)dw_s.\nonumber
\end{eqnarray}
Taking the inverse Fourier transform, we obtain
\begin{eqnarray}
\begin{array}{l}
v(t,x)=(\Ga(t,\cdot)*_x  u_0)(x)+\int^t_0 (\;\Ga(t-s,\cdot)*_x
f(s,\cdot)\;)(x)ds\vspace{0.3cm} \\ \qquad\quad\quad+\int^t_0 (\; \Ga(t-s,\cdot)*_x
g(s,\cdot)\;)(x)dw_s,\quad t>0,\;x\in\Rb,
\label{solution}
\end{array}
\end{eqnarray}
where $\Ga(t,x):= \frac{1}{(4\pi
t)^{\frac{n}{2}}}e^{-\frac{|x|^2}{4t}}I_{t>0}$ is the inverse
Fourier transform of $e^{- |\xi|^2 t}I_{t>0}$ and $*_x$ denotes
convolution on $x$. We restrict $v$ on $\Omega\times (0,T) \times
D$. Secondly, we find  the solution $h=h(\omega,t,x)$ of the
following simple (stochastic) initial-boundary value problem:
\begin{align}\label{heat}
\left\{\begin{array}{l}
h_{t}(\omega,t,x)= \De h(\omega,t,x), \quad (\omega,t,x)\in \Omega\times (0,T) \times D,\vspace{0.3cm}\\
h(\omega,t,x) =b(\omega,t,x) - v(\omega,t,x),\quad (\omega,t,x)\in\Omega\times (0,T) \times \pa D,\vspace{0.3cm}\\
h(\omega,0,x) = 0,\quad \omega\in\Omega,\;x\in D.
\end{array}
\right.
\end{align}
Then one can easily check that $u = v+ h$ is indeed a solution of
\eqref{main equation}. Since information of $h$ is well known, the
estimations of three parts of $v$ in \eqref{solution} are important
to us; especially the third one, the \emph{stochastic integral
part}.

We are to find a solution space for $u$ and the spaces for
$f,g,b,u_0$ so that the restriction of the three terms in the right
hand side of (\ref{solution}) on $\Omega\times (0,T)\times D$ and
$h$ belong to the solution space and moreover $u$ is unique in it.
We use two types of spaces in this paper; spaces of Bessel
potentials and Besov spaces.

In this paper we let $n\geq 2$, $0 < T < \infty$, and $D$ be a
bounded Lipschitz domain in $\Rb$. Denote
\begin{eqnarray}
D_T := (0,T)\times D,\quad  \pa D_T :=  (0,T)\times \pa D, \quad
\Rb_T :=(0,T)\times  \Rb.\nonumber
\end{eqnarray}
Also, we assume $2\le p<\infty$ instead of the usual deterministic
setup $1<p<\infty$ ; this restriction is due to the stochastic
part in (\ref{main equation}) (see \cite{Kr94}). The main result
in this paper is the following.
\begin{theo}\label{theo}
Let $2\le p<\infty$ and $\frac1p<k<1+\frac1p$. Assume $f \in
{\mathbb B}^{k -2, \frac12(k -2)}_{p,o}(D_T), \, g \in {\mathbb
B}^{k-1}_{p,o} (D_T), \,\,
 b \in {\mathbb B}^{k -\frac1p,\frac12(k-\frac1{p})}_p(\pa D_T)$
and $u_0 \in L^p(\Om, {\mathcal G}_0, \mathcal{U}^{k
-\frac2p}_p(D))$. If $ \frac3p< k<1+\frac{1}{p}$, we further
assume the compatibility condition $u_0(\om,x) = b(\om,0,x)$ for
$\om \in \Om$, $x \in \pa D$. Then
\begin{itemize}
\item[(1)] if $ \frac1p< k <1$, there is a unique solution $u\in
{\mathbb B}^{k,\frac12k}_p (D_T)$ of the initial boundary value
problem \eqref{main equation} such that
\begin{align}\label{equality}
\begin{array}{ll}
 \| u\|_{{\mathbb B}^{k,\frac12k}_p (D_T)} &  \leq c\Big(\|u_0\|_{ L^p(\Om, {\mathcal G}_0,  \mathcal{U}^{  k - \frac{2}p
}_p(D))} + \| f \|_{{\mathbb B}^{k -2,\frac12(k-2) }_{p,o}(D_T)}\\
& \quad\quad\quad   + \| g \|_{{\mathbb B}^{k-1}_{p,o} (D_T) } +
\| b \|_{{\mathbb B}^{k -\frac1p,\frac12(k -\frac1{p})}_p( \pa
D_T)}\;\; \Big),
\end{array}
\end{align}
where $c$ depends only on  $D, k,n,p,T$. \item[(2)] if $1 \leq k <
1 +\frac1p$, there is a unique solution $u\in {\mathbb B}^k_p(D_T)
$
 of the  problem \eqref{main equation} such that
\begin{align}\label{Mainequ}
\begin{array}{ll}
 \|u\|_{{\mathbb B}^k_p(D_T) }
&   \leq c\Big(\|u_0\|_{ L^p(\Om, {\mathcal G}_0,  \mathcal{U}^{
k - \frac{2}p
}_p(D))}+\|f\|_{{\mathbb B}^{k -2, \frac12( k-2)}_{p,o}(D_T) }\\
& \quad\quad\quad  + \| g \|_{{\mathbb B}^{k-1}_{p,o} (D_T) } + \|
b \|_{{\mathbb B}^{k -\frac1p,\frac12(k-\frac1{p})}_p(\pa
D_T)}\;\;\Big),
\end{array}
 \end{align}
where $c$ depends only on $D, k,n,p,T$.
\end{itemize}
\end{theo}
The explanation of spaces and notations appearing in Theorem
\ref{theo} is placed in Section \ref{Pre-1}.
\begin{rem}
1. In the part (1) of Theorem \ref{main equation} we estimate the
regularity of $u$ in $(t,x)$ simultaneously using anisotropic Besov
norm whereas in part (2) we focus on the regularity in $x$. As we
mentioned earlier, the regularity in time is limited while the one
in space is not.

2. If $g\equiv 0$ and $u_0\equiv 0$, then (1) of Theorem \ref{main
equation} coincides with \cite{JM}.
\end{rem}

We organized the paper in the following way. Section \ref{Pre-1}
explains spaces and notations. In Section 3 we place main lemmas
and the proof of Theorem \ref{theo}. The long proofs of some main
lemmas are located in Section \ref{20110324-1}, \ref{20110324-2},
\ref{20101216-1} and \ref{20101216-2}.

Throughout this paper we denote $A \thickapprox B$ when there are
positive constants $c_1$ and $c_2$ such that $c_1 A \leq B \leq
c_2 A$. Also, $A \lesssim B$ means that there is a positive
constant $c$ such that $A \leq c B$. All such constants depend
only on $n,k,p,T$ and the Lipschitz constant of $\partial D$. We
use the notations $a\vee b=\max\{a,b\}$, $a\wedge b=\min\{a,b\}$.

\mysection{Preliminaries}\label{Pre-1} \setcounter{equation}{0}

Throughout this paper we let $(\Om, {\mathcal G},\{{\mathcal
G}_t\}, P)$ be a probability space, where $\{{\mathcal G}_t\;|\; t
\geq 0\}$ be a filtration of $\si$-fields ${\mathcal G}_t \subset
{\mathcal G}$ with $\mathcal{G}_0$ containing all $P$-null subsets
of $\Om$. Assume that a one-dimensional $\{\mathcal G_t
\}$-adapted Wiener processes $w_{\cdot}$ is defined on $(\Om,
{\mathcal G}, P )$. We denote the mathematical expectation of a
random variable $X=X(\omega),\;\omega\in \Omega$ by $E[X]$ or
simply $EX$; we suppress the argument $\omega\in\Omega$ under the
expectation $E$.

For $k\in\mathbb{R}$ let ${H}^k_p(\Rb)$ be the space of Bessel
potential and $B^k_p(\Rb)$ be  the Besov space (see, for instance,
\cite{BL}, \cite{Tr83}). For later purpose we place a definition
of Besov spaces. Let $\hat f(\xi),\xi\in\mathbb{R}^n$ denote the
Fourier transform of $f(x),x\in\mathbb{R}^n$ and the space
${\mathcal S} ({\mathbb R}^{n})$ denote the \emph{Schwartz space}
on $\mathbb R^n$.
 Fix any $\phi\in {\mathcal S} ({\mathbb R}^{n})$  such that
$\hat{\phi}$ satisfies $\hat{\phi}(\xi)
> 0$ on $\frac12 < |\xi| < 2$, $\hat{\phi}(\xi)=0$ elsewhere, and
$\sum_{j=-\infty}^{\infty} \hat{\phi}(2^{-j}\xi) =1$ for $\xi \neq
0$. We define $\phi_j$ and $\psi$ so that their Fourier transforms
are given by
\begin{eqnarray}\label{psi1}
\begin{array}{ll}
\widehat{\phi_j}(\xi) &= \widehat{\phi}(2^{-j} \xi) \quad (j = 0, \pm 1, \pm 2 , \cdots)\\
\widehat{\psi}(\xi) & = 1- \sum_{j=1}^\infty \widehat{\phi} (2^{-j} \xi).
\end{array}
\end{eqnarray}
Then we define the Besov space $B^k_{p} ({\mathbb R}^{n})=
B^k_{p,p} ({\mathbb R}^{n})$ by
\begin{eqnarray*}
B^k_{p} ({\mathbb R}^{n}) = \{ f \in {\mathcal D}({\mathbb R}^{n})
\,\,\, |\, \,\,\, \|f\|_{B^k_{p}}:=\| \psi
* f\|_{L^p} + \Big[ \sum_{j=1}^{\infty} (2^{k j} \|\phi_j *
f\|_{L^p})^p\Big]^{\frac1p} < \infty \;\},
\end{eqnarray*}
where $ {\mathcal D}({\mathbb R}^{n})$ is dual space of Schwartz
space and  $*$ means the convolution.

\subsection{Spaces for $D$, $\partial D$ and $(0,T)$}\label{space}

When $k\geq 0$, we define
\begin{align*}
H^k_p(D) : = \{ F|_D \, \big| \, F \in H^k_p(\Rb)\}, \quad
B^k_p(D): = \{ F|_D \; \big| \; F \in B^k_p(\Rb)\},\quad
\textrm{resp.}
\end{align*}
with the norms
\begin{align*}
\| f\|_{H^k_p(D)} := \inf \| F\|_{H^k_p(\Rb)},\quad \|
f\|_{B^k_p(D)} := \inf \| F\|_{B^k_p(\Rb)},\quad \textrm{resp.},
\end{align*}
where the  infima are   taken over $F \in H^k_p(\Rb)$ or $F \in
B^k_p(\Rb)$ satisfying $ F|_D =f $. We also define ${B^k_{p,o}}(D) $
as the closure of $C^\infty_c(D)$  in $B^k_p(D)$.
\begin{rem}\label{rem2}
Let $k_0$ be a nonnegative integer. Then the followings hold.
\begin{itemize}
\item[(1)]\begin{align*} \|f\|_{H^{k_0}_p(D)}^p \thickapprox
\sum_{0 \leq |\be| \leq k_0} \|D^\be f\|_{L^p(D)}^p,
\end{align*}
where $D^\be = D_{x_1}^{\be_1} D_{x_2}^{\be_2} \cdots
D^{\be_n}_{x_n}$ for $\be =(\be_1, \be_2, \cdots, \be_n) \in
(\{0\}\cup {\mathbb N})^n$.

\item[(2)] For  $k\in (k_0,k_0 +1)$
\begin{align*}
   \|f\|^p_{B^k_{p}(D)} \thickapprox
\| f\|^p_{H^{k_0}_p(D) } +  \sum_{|\be| =k_0 }\int_{D} \int_{D }
\frac{|D^\be f(x) - D^\be f(y)|^p}{|x-y|^{n+p(k-k_0)}} dxdy.
\end{align*}
The spaces $B^k_p(\pa D)$, $k\in (0,1)$ are defined similarly.

\item[(3)] Let $k=k_0+\theta$ with $\theta\in (0,1)$. Then the
space $B^k_p(D)$ satisfies the following real interpolation
property $($see Section 2 of \cite{JK}\:$)$:
\begin{align}\label{interpolation}
(H^{k_0}_p(D), H^{k_0 +1}_p(D))_{\te, p} = B^{k}_p(D).
\end{align}

\end{itemize}
\end{rem}

\quad When $k<0$ we define $B^k_p(D)$ as the dual space of
${B}^{-k}_{q,o} (D)$ and ${B}^k_{p,o}(D)$ as the dual space of
$B^{-k}_q(D)$, i.e., $B^k_p(D)=({B}^{-k}_{q,o} (D))^*$,
${B}^k_{p,o}(D)=(B^{-k}_q(D))^*$ with $\frac{1}{p}+\frac{1}{q}=1$.

We define $H^{\frac12 k}_p(0,T)$, $B^k_{p}(0,T)$ and
$B^k_{p,o}(0,T)$ similarly.
\begin{rem}
By the subscript $o$ in ${B}^k_{p,o}(D)$ $(k<0)$ we mean that the
natural extension of any distribution in this space vanishes
outside $D$ in the following sense. Let $h\in
{B}^k_{p,o}(D)=({B}^{-k}_{q}(D))^*$. We define the extension
$\tilde{h}\in {B}^k_{p}(\mathbb{R}^n)$ of $h$ by
\begin{eqnarray}
<\tilde{h},\Phi>:=<h,\Phi|_D>,\quad \Phi\in
B^{-k}_q(\mathbb{R}^n);\nonumber
\end{eqnarray}
note that by the very definition of $B^{-k}_q(D)$ we have
$\Phi|_D\in B^{-k}_q(D)$ and $<h,\Phi|_D>$ is well defined; here
the condition that $D$ is Lipschitz is used. Then for any  $\Phi$
with its support outside $D$, then $<\tilde{h},\Phi>=0$. This
means that $\tilde h$ vanishes outside $D$. A similar reasoning
says that the extension of any distribution in $B^k_p(D)$ may not
vanish outside $D$ and hence we do not add the subscript $o$.
\end{rem}

\quad For the initial condition $u_0$ we need
\begin{eqnarray}\label{initial}
{\mathcal U}^k_p(D) :=
\left\{\begin{array}{l}\vspace{0.2cm}
{B}^k_p(D),\quad k \geq 0,\\
{B}^k_{p,o}(D),\quad k <0.
\end{array}
\right.
\end{eqnarray}

\subsection{Spaces for $D_T$, $\partial D_T$}\label{cylinder}
For $k \geq 0$ we define the anisotropic Besov space $
{B}^{k,\frac12k}_p(D_T)$ by
\begin{align}\label{besov}
{B}^{k,\frac12k}_p(D_T) :=  L^p\Big((0,T); B^k_p(D)\Big) \cap
L^p\left(D; B^{\frac{k}2}_p ((0,T))\right)
\end{align}
with the norm
\begin{align}
\| f \|_{{B}^{k,\frac12k}_p (D_T)}: = \Big(\int_0^T \|
f(t,\cdot)\|^p_{B^k_p(D)} dt \Big)^\frac1p + \Big(\int_D \|
f(\cdot, x)\|^p_{B^{\frac{k}2}_p((0,T))} dx
\Big)^\frac1p,\label{anisotropic norm}
\end{align}
where $B^{\frac{k}{2}}_p((0,T))$ is defined similarly as in
Section \ref{space}; we also define
\begin{eqnarray}
{{B}}^{k,\frac12k}_{p,o}(D_T) =  L^p\Big((0,T); {
B}^k_{p,o}(D)\Big) \cap L^p\Big(D; {B}^{\frac{k}2}_{p,o}
((0,T))\Big)\nonumber
\end{eqnarray}
with the same norm (\ref{anisotropic norm}).

For $k<0$ we define $ B^{k, \frac12k}_p(D_T) = (  {B}^{-k,
-\frac12k}_{q,o}(D_T))^*$ and ${B}^{k, \frac12k}_{p,o}(D_T) = (
B^{-k, -\frac12k}_q(D_T))^*$ with $\frac{1}{p}+\frac{1}{q}=1$.

We define ${B}^{k,\frac12k}_p(\pa D_T)$, ${B}^{k,\frac12k}_{p,o}(\pa
D_T)$, $k\in(0,1)$ similarly.
%We also define
%$\mathcal{B}^{k}_{p,0}(D_T)=\{ u \in \mathcal{B}^{k}_{p}(D_T)\;|\;u(0,x)=0,\;x\in D\}$,
%$\mathcal{B}^{k}_{p,0}(\pa D_T)=\{ u \in \mathcal{B}^{k}_{p}(\pa D_T)\;|\;u(0,x)=0,\;x\in \pa D\}$.

\subsection{Stochastic Banach spaces} The solution $u$ and
functions $f,g,b,u_0$ in (\ref{main equation}) are all random. Using
Section \ref{space} and \ref{cylinder} we construct the spaces for
them. We describe two types of spaces. The first type emphasizes the
regularity in $x$ whereas the second type does the regularity in
$t,x$ together. Again, let $k\in\mathbb{R}$.

We can consider $u,f,g,b$ as function space-valued stochastic
processes and hence $(\Omega\times(0,T),\mathcal{P},$
$P\bigotimes\ell((0,T]))$ is a suitable choice for their common
domain, where ${\mathcal P}$ is the predictable $\si$-field
generated by $\{{\mathcal G}_t : t \geq 0\}$ (see, for instance, pp.
84--85 of \cite{K1}) and $\ell((0,T])$ is the Lebesgue measure on
$(0,T)$. We define
\begin{align*}
&{\mathbb H}^k_p(\Rb_T) = L^p( \Om\times (0,T), {\mathcal P}, H^k_p
(\Rb)),\quad {\mathbb B}^k_p(\Rb_T) = L^p( \Om\times (0,T),
{\mathcal P}, B^k_p (\Rb))
\end{align*}
and the norms
\begin{align*}
\| f\|_{{\mathbb H}^k_p(\Rb_T)} = \Big( E \int_0^T \|
f(s,\cdot)\|^p_{H^{k}_p (\Rb)}ds \Big)^\frac{1}p, \quad  %\quad \|
%g\|_{{\mathbb H}^k_p(T, \mathbb R^m)}  = \Big( E \int_0^T \|
%g(s,\cdot)\|^p_{H^{k}_p (\Rb ; \mathbb R^m)}ds \Big)^\frac{1}p,\\
\| f\|_{{\mathbb B}^k_p(\Rb_T)} = \Big( E \int_0^T \|
f(s,\cdot)\|^p_{B^{k}_p (\Rb)}ds \Big)^\frac{1}p %\quad \|
%g\|_{{\mathbb B}^k_p(T, \mathbb R^m)}  = \Big( E \int_0^T \|
%g(s,\cdot)\|^p_{B^{k}_p (\Rb ; \mathbb R^m)}ds \Big)^\frac{1}p.
\end{align*}
; we  suppress $\omega$ in $f$. Similarly we define
\begin{eqnarray}
{\mathbb H}^k_p(D_T) = L^p( \Om\times (0,T], {\mathcal P}, H^k_p
(D)),\quad {\mathbb B}^k_p(D_T) = L^p( \Om\times (0,T), {\mathcal
P}, B^k_p (D)),\nonumber
\end{eqnarray}
\vspace{-1cm}
\begin{eqnarray}
{\mathbb B}^k_{p,o}(D_T) = L^p( \Om\times (0,T), {\mathcal P},
{B}^k_{p,o} (D)).\nonumber
\end{eqnarray}
%&{\mathbb H}^k_{p,0}(D_T) =\{u\in {\mathbb H}^k_p(D_T)| u|_{\{t=0\}}=0\}
%L^p( \Om\times (0,T], {\mathcal P}, H^k_{p,0}
%(D)), \quad
%{\mathbb B}^k_{p,0}(D_T) = L^p( \Om\times (0,T], {\mathcal P}, B^k_{p,0}
%(D)).

We also define the stochastic \emph{anisotropic Besov} spaces
\begin{align*}
&{\mathbb B}^{k,\frac12k}_p(D_T) = L^p( \Omega, {\mathcal G},
{B}^{k, \frac12k}_p(D_T)), \quad  {\mathbb B}^{k,\frac12k}_p(\pa
D_T) = L^p( \Omega, {\mathcal G}, {B}^{k, \frac12k}_p(\pa D_T))
\end{align*}
with norms
\begin{align*}
\| f\|_{{\mathbb B}^{k,\frac12k}_p(D_T)} = \left( E  \|
f\|_{B^{k,\frac12k}_p(D_T)}^p \right)^\frac{1}p,\quad \|
f\|_{{\mathbb B}^{k,\frac12k}_p(\pa D_T)} = \left( E  \|
f\|_{B^{k,\frac12k}_p(\pa D_T)}^p \right)^\frac{1}p.
\end{align*}
Similarly we define ${\mathbb B}^{k,\frac12k}_{p,o}(D_T) = L^p(
\Omega, {\mathcal G}, {B}^{k, \frac12k}_{p,o}(D_T))$.

\section{Lemmas and Proof of Theorem \ref{theo}}\label{Lemmas}
\setcounter{equation}{0}

In this section we estimate the three terms of (\ref{solution}) and
prove our main theorem.

For $l<0$, if $h \in {B}_{p,o}^{l}(D)=(B^{-l}_q(D))^*$, then we
define $\tilde h \in  B_p^{l}(\Rb)$ as the trivial extension of $h$
by
\begin{align}\label{extension1}
<\tilde h, \phi> : = <h, \phi|_D>, \quad \phi \in B^{-l}_q(\Rb),
\end{align}
; note $\|\tilde h\|_{B_p^{l}(\Rb)} \thickapprox
\|h\|_{{B}_{p,o}^{l}(D)}$. For $l\ge 0$, if $h \in  B_p^{l}(D)$,
then  we define $\tilde h \in B_p^{l}(\Rb)$ as the Stein's
extension of $h$ with $\|\tilde h\|_{B_p^l(\Rb)} \lesssim
\|h\|_{B_p^{l}(D)}$ (see section 2 of \cite{JK} and Chapter 6 of
\cite{St}); this extension is possible since our space domain $D$
is at least Lipschitz. Recall the definition of $\mathcal
U^l_p(D)$ in (\ref{initial}).
\begin{lemm}\label{prop1}
Let $0<k<2$. We assume $ u_0(\omega,\cdot) \in  \mathcal
U_p^{k-\frac2p} (D)$ for each $\omega\in\Omega$. Let $\tilde{u}_0$
denote the extension of $u_0$  (trivial or Stein's). For each
$(\omega,t,x)\in \Omega\times(0,T)\times \mathbb{R}^n$ define
\begin{align}\label{u-1}
v_1(\omega,t,x) := \left\{ \begin{array}{ll}
<\tilde u_0(\omega,\cdot), \Ga(t, x -\cdot)>, \qquad\;\;\;\textrm{if}\;\;  0\le k <\frac2p ,\\
\\
\int_{\Rb} \Ga(t, x-y) \tilde u_0(\omega,y)\; dy,\qquad\textrm{if}\;\;\frac2p \leq k.
\end{array}
\right.
\end{align}
Then $ v_1(\omega,\cdot,\cdot) \in {B}^{k,\frac12k }_p ( \Rb_T )$
for each $\omega$ and
\begin{align}\label{estimate 1}
\| v_1(\omega,\cdot,\cdot)\|_{{B}^{k, \frac12 k}_p ( \Rb_T )} \le
c\, \| u_0(\omega,\cdot)\|_{\mathcal U_p^{k -\frac2p} (D)},\quad
\omega\in\Omega,
\end{align}
where $c$ is independent of $u_0$ and $\omega$.
\end{lemm}
; the proof is presented in Section \ref{20110324-1}. \vspace{0.3cm}

For $0 < k <2$ and $ h=h(t,x) \in {B}^{ k-2, \frac12k
-1}_{p,o}(D_T)$ we define $\tilde h \in B^{ k-2, \frac12 k-1
}_p({\mathbb R}^{n+1})$ by
\begin{align}\label{extension2}
<\tilde h, \phi> : = <h,\phi|_{D_T}>,\quad \phi \in B^{ 2-k,1-
\frac12k }_q({\mathbb R}^{n+1}).
\end{align}
In this case $\|\tilde h\|_{B_p^{ k -2,\frac12 k -1}(\mathbb
R^{n+1})} \thickapprox \|h\|_{{B}_{p,o}^{k-2,\frac12 k-1}(D_T)}$.
\begin{lemm}\label{prop2}
Let $ 0 <k < 2$ and $f\in {\mathbb B}^{k-2,\frac12k-1}_{p,o}(D_T)$. Define
\begin{align}\label{u-2}
v_2(\omega,t,x) := < \tilde f(\omega,\cdot,\cdot), \Ga(t-\cdot,
x-\cdot)>.
\end{align}
Then  $v_2 \in {\mathbb B}^{k,\frac12k }_p(\Rb_T )$ and
\begin{align}\label{estimate 2}
\| v_2\|_{{\mathbb B}^{k,\frac12k }_p(\Rb_T )} \leq c \|
f\|_{{\mathbb B}^{k-2,\frac12k-1}_{p,o}(D_T)}.
\end{align}
\end{lemm}
; the proof is in Section \ref{20110324-2}. \vspace{0.3cm}

Before we estimate $v_3$ let us place the following lemma which is
Exercise 5.8.6 in \cite{BL}:
\begin{lemm}\label{interpolationlemma}
Assume that $A_0$ and $A_1$ are Banach spaces and that $1 \leq p <
\infty$, $0 < \te < 1$. Then
\begin{align*}%\label{interpolation}
(L_p(A_0), L_p (A_1))_{\te,p} = L_p ((A_0, A_1)_{\te,p}),
\end{align*}
where $(\cdot, \cdot)_{\te,p}$ is a real interpolation.
\end{lemm}
If  $ 0<k<1$, then for $ g=g(\omega,t,x) \in {\mathbb B}^{k
-1}_{p,o}(D_T)$ we define $\tilde g \in \mathbb B^{k-1}_p({\mathbb
R}^{n+1})$ by
\begin{align}%\label{extension3}
<\tilde g(\om,t, \cdot), \phi> : = <g(\om, t,
\cdot),\phi|_{D_T}>,\quad \phi \in B^{k-1}_q({\mathbb R}^{n})
\end{align}
and, if $ k \geq1$, we define $\tilde g(\om,t, \cdot) \in
{B}^{k-1}_p({\mathbb R}^{n+1})$ by  $\tilde g(\om, t,x)= g(\om,
t,x)$ for $x \in D$ and $\tilde g(\om,t,x) =0$ for $x \in \Rb
\setminus \bar D$. Then we get $\|\tilde g\|_{\mathbb
B_p^{k-1}(\mathbb R^{n+1})} \thickapprox \|g\|_{\tilde{\mathbb
B}_p^{k-1}(D_T)}$.

\begin{lemm}\label{prop3}
Let $k > 0$ and $g \in {\mathbb B}^{k -1}_{p,o}(D_T)$. Define
\begin{align}
v_3(t,x) := \left\{\begin{array}{ll} \int_0^t< \tilde g(s, \cdot),
\Ga(t-s, x-\cdot)> dw_s, \quad \textrm{if}\;\;\; 0 < k < 1,\\
\label{v3}\vspace{0.1cm}
\\
\int_0^t \int_{\Rb} \Ga(t-s, x-y) \tilde g(s,y) dy\;dw_s,\quad
\textrm{if}\;\;\; 1 \leq k
\end{array}
\right.
\end{align}
;we suppressed $\omega$. Then  $v_3 \in \mathbb B^{k }_p(\mathbb
R^n_T)$ with
\begin{align}\label{estimate 3}
\| v_3\|_{\mathbb B^{k }_p(\mathbb R^n_T)} \leq c \| g\|_{{\mathbb
B}^{k -1}_{p,o}(D_T)}.
\end{align}
\end{lemm}
\begin{proof}
Apply the result in  \cite{K1} and Lemma \ref{interpolationlemma}.
\end{proof}

For $\ep\in (0,1)$ we let $p_0 = \frac12 + \frac12\ep, \,\, p_0^{'} = \frac12 -\frac12\ep$.
We say that $(\frac1p, k) \in {\mathcal R}_\ep$ if  $\al$ and $p$ are numbers satisfying one of the followings:
\begin{itemize}
\item[1.]
$ p_0 <p < p_0^{'}$\; if\; $ 0 < k <1$,
\item[2.]
$1 < p \leq p_0$ \;if\; $\frac2p - 1-\ep < k <1$,
\item[3.]
$p_0^{'} \leq p < \infty$ \;if\; $0 <k <\frac2p +\ep  $.
\end{itemize}

%\begin{figure}[h]
%\centering
%\includegraphics[scale=.60]{pic.jpg}
%\caption{${\mathcal R}_\ep$}
%\end{figure}

\begin{lemm}\label{prop4}
There is a positive constant $\ep\in (0,1) $ depending only on Lipschitz constant of $\partial D$ such that if
$(\frac1p, k) \in {\mathcal R}_\ep$, then for all
$b' \in {\mathbb B}^{k,\frac12k}_{p}(\pa D_T)$ with
$b'(\omega,0,x) =0$ for $\omega\in\Omega, x \in \pa D$ if $k > \frac2p$.
Then there is a unique solution $h \in {\mathbb B}^{k +\frac1p,\frac12k+ \frac1{2p}}_{p}(D_T)$ of the problem
 $($\ref{heat}$)$  in $\Omega\times D_T$ with boundary value $b'$ in place of $b-v$  and $h(\omega,0,x) =0$ for
 $\omega\in\Omega,\;x \in D$ and it satisfies
\begin{align}\label{estimate 4}
\| h\|_{{\mathbb B}^{k +\frac1p,\frac12k +\frac1{2p}}_{p } (D_T)} \leq c
\| b'\|_{{\mathbb B}^{k,\frac12k }_p (\pa D_T)}.
\end{align}
If $D$ is a $C^1$-domain, then  we can take $\ep =1$.
\end{lemm}
\begin{proof}
Apply \cite{B}, \cite{B2} and \cite{JM} for each $\omega\in\Omega$.
\end{proof}

We need the following restriction theorem from \cite{C}:
\begin{lemm}\label{prop5}
Let $\frac1p < k < 1 + \frac1p$. Then for any $h=h(t,x) \in
{B}^{k,\frac12k}_p ( \Rb_T )$, we have $h|_{\pa D_T } \in {B}^{k
-\frac1p,\frac12k -\frac1{2p}}_p (\pa D_T )$.
\end{lemm}

The following lemma for the stochastic part $v_3$ in
(\ref{solution}) is important and we elaborate the proof in
Section \ref{20101216-1} and \ref{20101216-2}.
\begin{lemm}\label{maintheo}
Assume \;$2 \leq p < \infty$.
\begin{itemize}
\item[(1)]
Let $ \frac1p <k <1$ and  $g \in {\mathbb B}_{p,o}^{k-1}(D_T)$. Then
$v_3$ defined for such $k$ in Lemma \ref{prop3} belongs to ${\mathbb
B}^{k ,\frac12k   }_p(\Rb_T)$ and
\begin{align}\label{equality2}
\| v_3\|_{{\mathbb B}^{k ,\frac12k   }_p(\Rb_T)} \leq c \| g
\|_{{\mathbb B}^{k-1}_{p,o} (D_T)}.
\end{align}

\item[(2)] Let $1\le k < 1+ \frac1p$ and $g \in {\mathbb
B}^{k-1}_{p,o} (D_T) $. Then $v_3$ defined for such $k$  in Lemma
\ref{prop3} satisfies
\begin{align}\label{equality3}
\| v_3|_{\pa D_T}\|_{{\mathbb B}^{k - \frac1p,\frac12k
-\frac1{2p}}_p(\pa D_T)} \leq c \| g\|_{{\mathbb B}^{k-1}_{p,o}
(D_T)}
\end{align}
\end{itemize}
\end{lemm}
By Lemma \ref{prop1} - Lemma \ref{maintheo} the proof of Theorem
\ref{theo} follows.
%\begin{rem}\label{rem1}
%Using real interpolation and Theorem \eqref{maintheo}, we get that  $ v|_{\pa D_T } \in {B}^{\al+1 - \frac1p,\frac12\al +\frac12-\frac1{2p}}_p(\pa D_T)$ for $g \in L^p((0,T), {B}^\al_{p0})$.
%\end{rem}

{\bf Proof of Theorem \ref{theo}}\quad Recall the derivation of the
solution $u=v+h$ in Section \ref{introduction}.

(1) By Lemma \ref{prop1}, \ref{prop2} and Lemma \ref{maintheo} (1),
the (random) function $v:=v_1+v_2+v_3$ is in ${\mathbb B}^{k,\frac12
k }_{p}(\Rb_T)$; note that the definition of $u_1$ in Lemma
\ref{prop1} is different by the cases $k\in
(\frac{1}{p},\frac{2}{p})$ and $k\in [\frac{2}{p},1)$. Moreover, we
choose the definition of $u_3$ in Lemma \ref{prop3} for $k\in
(\frac{1}{p},1)$. Now, using Lemma \ref{prop5} for each
$\omega\in\Omega$, we have $b' := b - v|_{\pa D_T} \in {\mathbb
B}^{k-\frac1p,\frac12 k -\frac1{2p}}_{p}(\pa D_T)$. Let $u_4 \in
{\mathbb B}^{k,\frac12 k }_p(  D_T)$ be the unique solution of the
problem (\ref{heat}) which does exist by Lemma \ref{prop4}. Then
$u:=v+h$ is a solution of (\ref{main equation}) and the estimate
(\ref{equality}) follows (\ref{estimate 1}), (\ref{estimate 2}),
(\ref{equality2}) and (\ref{estimate 4}). The uniqueness of such $u$
follows the theory of deterministic heat equation.

(2)  Set $v$ as in (1) by choosing the appropriate definitions of
$v_1$, $v_3$ when $k\in [1,1+\frac1p)$.  Then proof is similar to
the case (1). However, this time we can not have $v_3$ in
${\mathbb B}^{k ,\frac12 k   }_p(\Rb_T)$ although it is in
$\mathbb B^{k }_p(\mathbb R^n_T)$ by the Lemma \ref{prop3}. Hence,
we have $v$ is in $\mathbb B^{k }_p(\mathbb R^n_T)$ as $v_1$,
$v_2$ are trivially in $\mathbb B^{k }_p(\mathbb R^n_T)$ (see
(\ref{besov})). Nevertheless, by using Lemma \ref{maintheo} (2) we
still have $b'\in {\mathbb B}^{k-\frac1p,\frac12 k
-\frac1{2p}}_{p}(\pa D_T)$.  By choosing $v_4$ as before  in
${\mathbb B}^{k,\frac12 k }_p( D_T)$ and hence $\mathbb B^{k
}_p(\mathbb R^n_T)$, we have a solution of (\ref{main equation})
in $\mathbb B^{k }_p(\mathbb R^n_T)$ and the estimate
(\ref{Mainequ}) follows (\ref{estimate 1}), (\ref{estimate 2}),
(\ref{equality3}), (\ref{estimate 4}) with (\ref{besov}). The
solution is unique.  \hspace{4.5cm}$\Box$

\mysection{Proof of Lemma \ref{prop1}}\label{20110324-1}
 We believe that
one may find a proof of Lemma \ref{prop1} is in the literature.
However, we can not find the exact reference and, hence, we
provide our own proof. We start with a lemma for multipliers.
\begin{lemm}\label{multiplier2}
Let $\Phi (\xi) = \hat \phi(2^{-1} \xi) + \hat \phi (\xi) + \hat
\phi (2\xi) $ with $\phi$ in the definition of Besov spaces, $\Phi_j
(\xi) = \Phi(2^{-j} \xi)$, and $\rho_{tj}(\xi) = \Phi_j ( \xi)
e^{-t|\xi|^2}$ for each integer $j$. Then $\rho_{tj}( \xi)$ is a
$L^p(\mathbb{R}^n)$-multiplier with the finite norm $M(t,j)$ for $1
< p <\infty$. Moreover for $t
> 0$
\begin{align}\label{multiplier2_2}
M(t,j) & \lesssim e^{-\frac14 t2^{2 j}}\sum_{0 \leq i \leq n} t^i
2^{2 ij} \lesssim e^{-\frac18 t2^{2 j}}.
\end{align}
\end{lemm}
\begin{proof}
The $L^p(\mathbb{R}^n)$-multiplier norm $M(t,j)$ of $\rho_{tj}(\xi)
$ is equal to the $L^p(\mathbb{R}^n)$-multiplier norm of
$\rho_{tj}^{'}(\xi)  := \Phi(\xi) e^{-t2^{2 j} |\xi|^2}$ (see
Theorem 6.1.3 in \cite{BL}). Now, we make use of the Theorem 4.6\'{}
of \cite{St}. We  assume $\be_1,  \be_2, \cdots, \be_l =1$ and
$\be_i =0$ for $l+1 \leq i \leq n$, and  set $\be=(\be_1, \be_2,
\cdots, \be_n)$. Since $supp\;(\Phi) \subset \{\xi \in \mathbb{R}^n
\, | \, \frac14 < |\xi| < 4 \}$, we have
\begin{align*}
|D^\be_{\xi} \rho_{tj}^{'}(\xi)| &\lesssim \sum_{0 \leq i \leq
|\be|} t^i 2^{2 ij} e^{-\frac14t 2^{2 j}} \chi_{\frac14 < |\xi| < 4}
(\xi),
\end{align*}
where $\chi_A$ is the characteristic function on a set $A$. Hence,
for $A = \prod_{1 \leq i \leq l} [2^{k_i}, 2^{k_i +1}]$ we receive
\begin{align*}
\int_A \left|\frac{\pa^{|\be|} }{\pa \xi_{\be}}
\rho^{'}_{tj}(\xi)\right| d\xi_\be \leq c \sum_{0 \leq i \leq n} t^i
2^{2 ij} e^{-\frac14t 2^{2 j}}.
\end{align*}
\end{proof}

Below $\tilde u_0 $ is the extension of $u_0$; note
$\tilde{u}_0(\omega,\cdot)\in B^{k-\frac2p}_p(\mathbb{R}^n)$ for
each $\omega\in \Omega$. The following lemma handles the case
$k=0$.
\begin{lemm}\label{frac3p}
We have
\begin{align}\label{boundary2}
\|v_1(\omega)\|_{L^p(\mathbb{R}^n_T)} \leq c \;\|
\tilde{u}_0(\omega)\|_{{B}^{-\frac{2}p}_p (\mathbb{R}^n)},\quad
\omega\in \Omega,
\end{align}
where the constant $c$ is independent of $u_0$ and $\omega$.
\end{lemm}
\begin{proof}
We may assume that $\tilde{u}_0  \in C^\infty_0 (\mathbb{R}^n)$
since $C^\infty_0(\mathbb{R}^n)$ is dense in ${B}^{-\frac{2}p}_p(\mathbb{R}^n)$.
We use the dyadic partition of
unity $\hat{\psi} (\xi) + \sum_{j=1}^{\infty} \hat \phi(2^{-j} \xi)
=1$ for $\xi \in \mathbb{R}^n$, so that we can write
\begin{align*}
\hat v_1(t,\xi) = \hat\psi (\xi) e^{-t|\xi|^2}\widehat{ \tilde
u_0}(\xi)+ \sum_{j=1}^\infty \hat\phi (2^{-j} \xi) e^{-t|\xi |^2}
\widehat{ \tilde{u}_0}(\xi).
\end{align*}
For $t>0$ we have
\begin{eqnarray}
\begin{array}{ll} \label{u_1} \vspace{2mm}
&\int_{\mathbb{R}^n}  | v_1(t,x)|^pdx \\ \vspace{2mm} &  \leq
\int_{\mathbb{R}^n} \left|{\mathcal F}^{-1} \left(e^{-t|\xi|^2}
\hat\psi(\xi)\;\widehat{ \tilde u_0}(\xi)\right)(x)\right|^p dx +
\int_{\mathbb{R}^n} \left|{\mathcal F}^{-1} \Big(\sum_{j=1}^\infty
e^{-t|\xi|^2} \hat \phi_j(\xi) \;\widehat{ \tilde u_0}(\xi)
\Big)(x)\right|^p dx.
\end{array}
\end{eqnarray}
The first term on the right-hand side of \eqref{u_1} is dominated by
\begin{align}\label{negative2}
\|\psi * \tilde u_0\|^p _{L^p (\mathbb{R}^n)}.
\end{align}
Now, we estimate the second term on the right-hand side of
\eqref{u_1}. We use the facts that $\hat \phi_j = \Phi_j \hat
\phi_j$  for all $j$, where  $\Phi_j$ is defined in Lemma
\ref{multiplier2}. By Lemma \ref{multiplier2}, $\ \Phi_j( \xi) e^{-t
|\xi|^2}s$ are the $L^p(\mathbb{R}^n)$-Fourier multipliers with the
norms $M(t,j)$. Then we divide the sum as
\begin{align*}
&\int_{\mathbb{R}^n} \left|{\mathcal F}^{-1} \Big(\sum_{j=1}^\infty
e^{-t|\xi|^2} \hat \phi_j(\xi)\; \widehat{ \tilde
u_0}(\xi) \Big)(x)\right|^p dx \\
&= \int_{\mathbb{R}^n} \left|{\mathcal F}^{-1}
\Big(\sum_{j=1}^\infty \Phi_j( \xi) e^{-t|\xi|^2} \hat
\phi_j(\xi)\;\widehat{ \tilde
u_0}(\xi) \Big)(x)\right|^p dx\\
& \leq \Big(\sum_{2^{2j}\le1/t} M(t,j) \| \tilde u_0 *
\phi_j\|_{L^p} \Big)^p
+ \Big(\sum_{2^{2j}\ge1/t} M(t,j) \| \tilde u_0 *\phi_j\|_{L^p} \Big)^p\\
& =: I_1(t) + I_2(t).
\end{align*}

%We can write
%\begin{align*}
%\hat v_1(t,\xi)
%&= e^{-t|\xi|^2} \left(\hat{\psi}(\xi) + \sum_{j=1}^{\infty} \hat\phi_j(\xi)\right) \widehat{ \tilde u_0}(\xi) \\
%&= e^{-t|\xi|^2}\hat\psi (\xi) \widehat{ \tilde u_0}(\xi)+ \sum_{j=1}^\infty e^{-t|\xi |^2} \hat\phi_j (\xi) \widehat{ \tilde{u}_0}(\xi).
%\end{align*}
%Note that $\hat \phi_j = \Phi_j \hat \phi_j$  for all $j$, where  $\Phi_j$ is defined in Lemma \ref{multiplier2}.
%By Lemma \ref{multiplier2}, $\ \Phi_j( \xi) e^{-t|\xi|^2}$ is the $L^p(\mathbb{R}^n)$-Fourier multiplier with the norm $M(t,j)$.
%Hence we have
%\begin{align*}
%\|v_1\|_{L^p(\mathbb{R}^n_T)}
%&\le \|{\mathcal F}^{-1}\left(e^{-t|\xi|^2}\hat\psi (\xi) \widehat{ \tilde u_0}(\xi)\right)\|_{L^p(\mathbb{R}^n_T)} \\
%&+ \|\sum_{j=1}^{\infty}{\mathcal F}^{-1}\left(\Phi_j(\xi) e^{-t|\xi|^2}\hat\phi_j (\xi) \widehat{ \tilde u_0}(\xi)\right)\|_{L^p(\mathbb{R}^n_T)} \\
%&\le \left(\int_0^T \|\psi*\tilde u_0\|_{L^p(\mathbb{R}^n)}^p dt\right)^{1/p} \\
%&+ \left(\int_0^T \left(\sum_{j=1}^{\infty} M(t,j)\|\phi_j*\tilde u_0\|_{L^p(\mathbb{R}^n)}\right)^p dt\right)^{1/p}.
%\end{align*}
%Now, we divide the last sum as $I_1(t)+I_2(t)$, where
%\begin{align*}
%I_1(t) &= \sum_{2^{2j}\le1/t} M(t,j) \| \tilde u_0*\phi_j\|_{L^p} \\
%I_2(t) &= \sum_{2^{2j}\ge1/t} M(t,j) \| \tilde u_0*\phi_j\|_{L^p}.
%\end{align*}

By Lemma \ref{multiplier2} we have $ M(t,j) \leq c$  for $t 2^{2 j}
\leq 1$. We  take $a$ satisfying $ -\frac{2}p < a < 0$ and then use
H\"older inequality to get
\begin{align*}
\int_0^T I_1(t)^p dt
& \lesssim \int_0^T  \Big(\sum_{2^{2j}\le1/t} 2^{-\frac{p}{p-1}aj} \Big)^{p-1} \sum_{2^{2j}\le1/t} 2^{paj} \|\phi_j*\tilde u_0\|^p_{L^p}dt \\
& \lesssim  \int_0^T t^{\frac12 pa } \sum_{2^{2j}\le1/t} 2^{paj} \|\phi_j*\tilde u_0\|^p_{L^p}dt\\
& \lesssim \sum_{j=1}^\infty 2^{paj} \|\phi_j*\tilde u_0\|^p_{L^p} \int_0^{2^{-2 j}}  t^{\frac12 pa }dt\\
& = c \sum_{j=1}^\infty 2^{-2 j } \|\phi_j*\tilde u_0\|^p_{L^p}.
\end{align*}
By Lemma \ref{multiplier2} again $ M(t,j) \leq c (t2^{2 j})^{-m}
\sum_{0 \leq i \leq n}(t 2^{2 j})^i \leq c 2^{(2 n-2 m)j} t^{n-m} $
for $t\cdot 2^{2 j} \geq 1$ and $m>0$. We fix $b>0$ and then choose
$m$ satisfying $ p\,2(n-m) + \frac12 p\,b +1 < 0$, so that we obtain
\begin{align*}
\int_0^T I_2(t)^p dt
&\lesssim \int_0^T    \Big(\sum_{2^{2j}\ge1/t} 2^{(2 n- 2 m)j} t^{n-m} \|\phi_j*\tilde u_0\|_{L^p} \Big)^pdt \\
& \lesssim \int_0^\infty t^{p ( n -  m) } \Big(\sum_{2^{2j}\ge1/t} 2^{-\frac{p}{p-1} bj} \Big)^{p-1} \sum_{2^{2j}\ge1/t}
2^{pbj}2^{p(2 n-2 m)j} \|\phi_j*\tilde u_0\|^p_{L^p}dt \\
& \lesssim \int_0^\infty t^{ p(  n- m )  + \frac12 pb} \sum_{2^{2j}\ge1/t} 2^{pbj}2^{p(2 n-2 m)j} \|\phi_j*\tilde u_0\|^p_{L^p}dt\\
& \lesssim \sum_{j=1}^\infty 2^{pbj}2^{p(2 n-2 m)j} \|\phi_j*\tilde u_0\|^p_{L^p} \int_{2^{-2j}}^\infty t^{p ( n - m )  + \frac12 pb}dt\\
& =c \sum_{j=1}^\infty 2^{-2 j  } \|\phi_j*\tilde u_0\|^p_{L^p}.
\end{align*}
\end{proof}

{\bf Proof of Lemma \ref{prop1}} \;\; The following is a classical
result (see \cite{LSU67}):
\begin{align}\label{w-2}
 \int_0^T  \| v_1(\omega,t, \cdot)\|_{H^2_p(\mathbb R^n)}^p dt + \int_{\mathbb R^n} \| v_1(\omega,\cdot, x)\|_{H^1_p(0,T)}^p dx \leq c\| \tilde u_0(\omega)\|^p_{B^{2-\frac2p}_p( \mathbb
 R^n)}, \quad \omega\in\Omega.
\end{align}
Using (\ref{w-2}), Lemma \ref{frac3p} and the following real
interpolations
\begin{eqnarray*}
(L^p (\mathbb R^n), H^{2}_p (\mathbb R^n))_{\frac{k}{2},p} = B^k_p
(\mathbb R^n),&&\hspace{-0.5cm} (L^p ((0,T)), H^1_p
((0,T))_{\frac{k}{2},p} =
B^{\frac{k}{2}}_p ((0,T)),\\
(B^{-\frac2p}_p(\mathbb{R}^n),B^{2-\frac2p}_p(\mathbb{R}^n))_{\frac{k}{2},p}&=&B^{k-\frac2p}_p(\mathbb{R}^n),
\end{eqnarray*}
we have
\begin{align*}
   \|v_1(\omega)\|^p_{B^{k,\frac12 k}(\mathbb{R}^n_T)}=
\int_0^T  \| v_1(\omega,t, \cdot)\|_{B^{k}_p(\mathbb R^n)}^p dt +
\int_{\mathbb R^n} \| v_1(\omega,\cdot,
x)\|_{B^{\frac{k}{2}}_p(0,T)}^p dx
   \leq c\| \tilde u_0(\omega)\|^p_{B^{k-\frac2p}_p( \mathbb R^n)}.
\end{align*}
This implies Lemma \ref{prop1}. \hspace{11cm}$\Box$

\mysection{Proof of Lemma \ref{prop2}}\label{20110324-2}

We need the space of the \emph{parabolic} Bessel potentials.
 For $l \in {\mathbb R}$ the
parabolic Bessel potential  $\Pi_l $ is  a distribution whose
Fourier transform  in ${\mathbb R}^{n+1} $ is defined by
\begin{eqnarray*}
\widehat{\Pi_{l}} (\tau, \xi) = c_k(1 +i\tau+
|\xi|^2)^{-\frac{l}{2}}, \quad \tau \in {\mathbb R}, \, \xi \in
\Rb.
\end{eqnarray*}
In particular, if $l > 0$, then
\begin{eqnarray}\label{bessel}
\Pi_l (t,x) = \left\{ \begin{array}{ll}
             c_l\; t^{\frac{l-n-2}{2}}\;e^{-t}\;e^{-\frac{|x|^2}{4t}} &
             \mbox{if       } t>0,\\\quad\\
               0  & \mbox{if       } t\leq 0
\end{array}
\right.
\end{eqnarray}
; see \cite{Jones68}. In particular, $\Pi_2=e^{-t}\Gamma$, where
$\Gamma$ is the heat kernel introduced in Section
\ref{introduction}.

For $1\leq p< \infty$ we define the space of the parabolic Bessel
potentials, ${ H}^{l,\frac12 l}_p ({\mathbb R}^{n+1})$, by
\begin{eqnarray*}
{H}^{l,\frac12 l}_p ({\mathbb R}^{n+1}) = \{ f \in {\mathcal
S}'({\mathbb R}^{n+1}) \; | \; \Pi_{-l}*f       \in L^p ({\mathbb
R}^{n+1}) \}
\end{eqnarray*}
with the norm $$\|f\|_{H^{l,\frac12 l}_p ({\mathbb R}^{n+1})} = \|
\Pi_{-l}*f  \|_{L^p({\mathbb R}^{n+1})},$$
 where
$*$ in this case is a convolution in ${\mathbb R}^{n+1}$ and
${\mathcal S}'({\mathbb R}^{n+1})$
 is the dual space of the Schwartz space
${\mathcal S}({\mathbb R}^{n+1})$. Note that if $l \geq 0$, we
have
\begin{align*}
 H^{l,\frac12 l}_p(\mathbb R^{n+1}) =
 L^p\Big({\mathbb R}; H^l_p(\Rb)\Big) \cap L^p\left(\Rb; H_p^{\frac12 l}({\mathbb R})\right).
 \end{align*}
For $l \geq 0$ we define
$$
H^{l,\frac12 l}_p({\mathbb R}^{n}_T) := \{ f|_{{\mathbb R}^{n}_T}
\;\; | \;\; f \in H^{l,\frac12 l}_p({\mathbb R}^{n+1}) \}
$$
and let $H^{l,\frac12 l}_{p,o} ({\mathbb R}^{n}_T)$ be the closure
of $C^{\infty}_c({\mathbb R}^{n}_T)$ in $H^{l,\frac12
l}_p({\mathbb R}^{n}_T)$.

For $l<0$ we also define $H^{l,\frac12 l}_{p}({\mathbb R}^{n}_T)$
and $H^{l,\frac12 l}_{p,o} ({\mathbb R}^{n}_T)$ as the dual spaces
of $H^{-l,-\frac12 l}_{q,o}({\mathbb R}^{n}_T)$ and
$H^{-l,-\frac12 l}_{q}({\mathbb R}^{n}_T)$ respectively with
$\frac1p+\frac1q=1$; $H^{l,\frac12 l}_{p}({\mathbb R}^{n}_T) =
(H^{-l,-\frac12 l}_{q,o}({\mathbb R}^{n}_T))^*, \,\, H^{l,\frac12
l}_{p,o}({\mathbb R}^{n}_T) = (H^{-l,-\frac12 l}_{q}({\mathbb
R}^{n}_T))^*$.

\vspace{0.3cm}

 {\bf Proof of Lemma \ref{prop2}} \quad We assumed $0<k<2$ and $f\in \mathbb{B}^{k-2,\frac12k-1}_{p,o}(D_T)$.
 Let $\tilde f$ is the extension of $f$ on ${\mathbb R}^{n+1}$.

1. We just show the case $k=0$
\begin{eqnarray}\label{part}
\| u_2(\omega)\|_{L^p({\mathbb R}^{n}_T)} \le c\| \tilde
 f(\omega)\|_{H^{-2,-1}_{p} ({\mathbb R}^{n+1})},\quad
\omega\in\Omega.
\end{eqnarray}
Then the classical result (\cite{LSU67}):
\begin{eqnarray}
\| u_2(\omega)\|_{H^{2,1}_p({\mathbb R}^{n}_T)} \leq c \| \tilde
f(\omega)\|_{L^p({\mathbb R}^{n+1})},\quad \omega\in\Omega \nonumber
\end{eqnarray}
and the real interpolations
\begin{eqnarray}
(L^p(\mathbb{R}^n_T),H^{2,1}_p(\mathbb{R}^n_T))_{\frac{k}{2},p}=B^{k,\frac12
k}_p(\mathbb{R}^n_T),\quad (H^{-2,-1}_p(\mathbb{R}^{n+1}),
L^p(\mathbb{R}^{n+1}))_{\frac{k}{2},p}=B^{k-2,\frac12
k-1}_p(\mathbb{R}^{n+1})\nonumber
\end{eqnarray}
lead us to
\begin{eqnarray}
\| u_2(\omega)\|_{{B}^{k,\frac12k}_p({\mathbb R}^{n}_T)} \le  c\|
\tilde f(\omega)\|_{{B}^{k-2,\frac12k -1}_{p}({\mathbb
R}^{n+1})},\quad \omega\in\Omega
\end{eqnarray}
and (\ref{estimate 2}) follows.

2. Since $ C^\infty_c(\mathbb R^n_T)$ is dense in $H^{l,\frac12
l}_{p,o} ({\mathbb R}^{n}_T)$ even for $l<0$, we may assume
$\tilde{f}$ is in $ C^\infty_c ({\mathbb R}^{n}_T)$. In this case
the representation
$$
u_2(\omega,t,x)=\int_0^t\int_{\Rb} \Ga(t-s,x-y)
\tilde{f}(\omega,s,y)\;dy\;ds
$$
is legal. Recalling $\Pi_2(t,x) = e^{-t} \Ga(t,x)$, we have
\begin{align*}
u_2(\omega,t,x) &= \int_0^t\int_{\Rb} e^{t-s}\Pi_2(t-s,x-y) \tilde{f}(\omega,s,y)\;dy\;ds % = e^t\int_0^t\int_{\Rb} H_2(t-s,x-y) g(s,y)dyds\\
  = e^{t}( \Pi_2*g(\omega,t,x)),
\end{align*}
where $g(\omega,s,y) = e^{-s} \tilde{f}(\omega,s,y)$. Hence,
\begin{align*}
\int_0^T\int_{\Rb} |u_2(\omega,t,x)|^p dxdt &= \int_0^T\int_{\Rb} e^{pt}|\Pi_2* g(\omega,t,x)|^pdxdt\\
        &\leq e^{pT} \int_{0}^\infty \int_{\Rb}|\Pi_2 *  g(\omega,t,x)|^pdxdt\\
        & \le e^{pT} \| g(\omega)\|^p_{H^{-2,-1}_p({\mathbb R}^{n+1})}\\
        & \lesssim e^{pT} \| \tilde{f}(\omega)\|^p_{H^{-2,-1}_p({\mathbb
        R}^{n+1})},
\end{align*}
where the last inequality follows by
\begin{eqnarray}
|<g(\omega),\phi>|=|<\tilde f(\omega),e^{-t}\phi>|\le \|\tilde
f(\omega)\|_{H^{-2,-1}_p({\mathbb
        R}^{n+1})}\;\|e^{-t}\phi\|_{H^{2,1}_p({\mathbb
        R}^{n+1})},\quad \phi\in H^{2,1}_p({\mathbb
        R}^{n+1})\nonumber
\end{eqnarray}
and the fact $\|e^{-t}\phi\|_{H^{2,1}_p({\mathbb
        R}^{n+1})}\lesssim \|\phi\|_{H^{2,1}_p({\mathbb
        R}^{n+1})}$.
We have received (\ref{part}) and the lemma is proved.
\hspace{7cm} $\Box$

\mysection{Proof of Lemma \ref{maintheo} (1)}  \label{20101216-1}

We only need to prove the case $T=1$:
\begin{eqnarray}
E \int_{\Rb} \int_0^1 \int_0^1 \frac{| v_3(t,x) -
v_3(s,x)|^p}{|t-s|^{1 + \frac{p}2k}}dsdt\;dx \lesssim \| \tilde
g\|^p_{L^p (\Om \times (0,1), \mathcal P, B^{k-1}_p
(\Rb))},\label{20110614-1}
\end{eqnarray}
where $\tilde g$ is the extension of $g$ and $v_3$ is defined in
(\ref{v3}) using $\tilde g$. Then the general case follows a
scaling argument with the fact that under the expectation we can
use any Brownian motion in the definition of $v_3$ and the
observation that $\bar w_r:=\frac{1}{\sqrt{T}}\,w_{\sqrt{T}\, r}$,
$r\in[0,1]$ is also a Brownian motion. Indeed, let\; $\tilde
g(\omega,r,y),\;\omega\in\Omega,\; r\in[0,T],\;y\in \mathbb{R}^n$
be given. Notice that we may assume that $\tilde g$ is smooth in
$y$. In this case
\begin{eqnarray}
v_3(t,x)&=&\int^t_0\int_{\mathbb{R}^n}<\Gamma (t-r,x-\cdot),\tilde
g(r,\cdot)>\,dw_r\nonumber\\
&=&\int^t_0\int_{\mathbb{R}^n}\Gamma (t-r,x-y)\tilde
g(r,y)dy\,dw_r,\quad\;\;t\in[0,T],\;x\in \mathbb{R}^n.\nonumber
\end{eqnarray}
Define $\bar v_3(t,x) = v_3(Tt, \sqrt{T} x)$, $t\in[0,1]$ and
$\bar{\tilde g} (r,y) = \tilde g(T r, \sqrt{T} y)$, $r\in [0,1]$.
Note
\begin{align*}
\bar v_3(t,x)
&= \int_0^{T t} \int_{\Rb} \Ga( Tt-r,\sqrt{T} x-y) \tilde g(r, y) dy \;dw_r \\
&= \sqrt{T} \int_0^{t} \int_{\Rb} (\sqrt{T})^n\Ga(Tt-Tr,\sqrt{T} x-\sqrt{T} y) \bar{\tilde g}(r, y) dy \;d \bar w_r\\
& = \sqrt{T} \int_0^{t} \int_{\Rb} \Ga(t-r,x-y) \bar{\tilde g}
(s,y) dy d \bar w_r.
\end{align*}
%Plug $\bar v_3,\bar{\tilde g}$ into (\ref{20110614-1}) in places
%of $v_3,\tilde{g}$. Then the left-hand side of (\ref{20110614-1})
%is
By obvious scaling and (\ref{20110614-1}) we receive
\begin{eqnarray}
&&\hspace{1.5cm}E \int_{\Rb} \int_0^T \int_0^T \frac{|v_3(t,x) -
v_3(s,x)|^p}{|t-s|^{1 + \frac{p}2k}}dsdt\;dx\nonumber\\
&=& T^{1 -\frac{p}2k + \frac{n}2}\; E \int_{\Rb} \int_0^1 \int_0^1
\frac{|\bar v_3 (t, x) -\bar v_3 (s,x)|^p}{|t-s|^{1 +
\frac{p}2k}}dsdt\;dx.\nonumber\\
&\lesssim& T^{1 -\frac{p}2k + \frac{n}2}\; \| \bar{\tilde
g}\|^p_{L^p (\Om \times (0,1), \mathcal P, B^{k-1}_p
(\Rb))}.\label{20110614-2}
\end{eqnarray}
To dominate (\ref{20110614-2}) by $\|\tilde g\|^p_{L^p (\Om \times
(0,T), \mathcal P, B^{k-1}_p (\Rb))}$ we observe the following.
Given a smooth function $f=f(y)$ define
$f_{\sqrt{T}}(y)=f\left(\sqrt{T}\,y\right)$. Then for any $\phi\in
C^{\infty}_0(\mathbb{R}^n),\;\|\phi\|_{B^{1-k}_q(\mathbb{R}^n)}=1$
with $\frac 1p+\frac 1q=1$,
%\|f_{\sqrt{T}}\|_{B^{k-1}(\mathbb{R}^n)}
\begin{eqnarray}
\int_{\mathbb{R}^n} f_{\sqrt{T}}(y)\phi(y)dy&=& T^{-\frac
n2}\int_{\mathbb{R}^n}
f(y)\phi_{\frac{1}{\sqrt{T}}}(y)dy\nonumber\\
&\le& T^{-\frac
n2}\|f\|_{B^{k-1}_p(\mathbb{R}^n)}\,\|\phi_{\frac{1}{\sqrt{T}}}\|_{B^{1-k}_q(\mathbb{R}^n)}\nonumber\\
&\le&T^{-\frac n2}\|f\|_{B^{k-1}_p(\mathbb{R}^n)}\cdot T^{\frac
n2}(1\vee
T^{-p(1-k)})\,\|\phi\|_{B^{1-k}_q(\mathbb{R}^n)}\nonumber\\
&\le& (1\vee
T^{-p(1-k)})\,\|f\|_{B^{k-1}_p(\mathbb{R}^n)};\nonumber
\end{eqnarray}
see Remark \ref{rem2} (2) for the second inequality. Hence,
$\|f_{\sqrt{T}}\|^p_{B^{k-1}_p(\mathbb{R}^n)}\le (1\vee
T^{k-1})\,\|f\|^p_{B^{k-1}_p(\mathbb{R}^n)}$. This and another
simple scaling imply that (\ref{20110614-2}) is indeed bounded by
$c\| \tilde g\|^p_{L^p (\Om \times (0,T), \mathcal P, B^{k-1}_p
(\Rb))}$, where $c$ depends only on $p,n,k,T$.

We need two more lemmas to prove Lemma \ref{maintheo} (1) with
$T=1$. The proof of the following lemmas are placed at the end of
this section.
\begin{lemm}\label{20110324-3}
Let $\frac1p<k<1$, $p\ge 2$ and $\tilde g\in \mathbb{H}^{k-1}_p
(\Rb_1) $. Then for $i=-1,-2,\ldots$ we have
\begin{align}\label{main3}
E \int_{\Rb} \int \int_{ 4^i\leq |t -s| \leq 4^{i+1}}
\frac{|v_3(t,x) - v_3(s,x)|^p}{|t-s|^{1 + \frac{p}2k}}dsdt\;dx
\lesssim \|\tilde g\|^p_{L^p ( \Om \times (0,1), \mathcal P,
H^{k-1}_p (\Rb))}.
\end{align}
\end{lemm}
\vspace{0.3cm}

Let $X_0$ and $X_1$ be a couple of Banach spaces continuously
embedded in a topological vector space and let $Y_0$ and $Y_1$ be
another such couple. We denote the real interpolation spaces
\begin{eqnarray}\label{intermediate}
X_{\te q} := (X_0, X_1 )_{\te ,q}, \,\,\, Y_{\te q} := ( Y_0, Y_1
)_{\te, q}, \quad 0 < \te < 1, \quad  1 \leq q \leq \infty
\end{eqnarray}
and the following well known result (see Theorem 1.3 in \cite{P2}):
\begin{lemm}\label{lemm4}
Let $T = \sum_{-\infty}^{\infty} T_i$, where $T_i : X_\nu \ri
Y_\nu $
 are bounded linear operators with norms  $M_{i,\nu}$
 such that $M_{i,\nu} \leq c \om^{i(\te - \nu)}, \nu = 0 ,1,$
  for some fixed  $\om\neq 1$ and $0 < \te < 1$.
 Then $T :X_{\te1} \ri Y_{\te\infty}$
 is a bounded linear operator.
\end{lemm}

Let us denote $S\tilde g:=v_3$.

{\bf Proof of Lemma \ref{maintheo} (1)} \quad 1. As we discussed,
it is enough to consider the case $T=1$. Recall $\frac1p < k <1$
and $p\ge 2$. Note that the extension $\tilde g$ of $g$ is in $L^p
(\Om \times (0,1), \mathcal P, B^{k-1}_p (\Rb))$. Since the random
function $S\tilde g$ belongs to $\mathbb B^{k}_p (\Rb_1)$ and
satisfies (\ref{estimate 3}) (Lemma \ref{prop3}), to prove
(\ref{equality2}) we only need to show
\begin{eqnarray}
E \int_{\Rb} \int_0^1 \int_0^1 \frac{|S\tilde g(t,x) - S\tilde
g(s,x)|^p}{|t-s|^{1 + \frac{p}2k}}dsdt\;dx \lesssim \| \tilde
g\|^p_{L^p (\Om \times (0,1), \mathcal P, B^{k-1}_p
(\Rb))}\label{tong}
\end{eqnarray}
; see (\ref{anisotropic norm}) and the time version of  Remark
\ref{rem2} (2). We follows the outline of \cite{JW}.

2. Define the space $Y$ whose element $h:\Om \times (0,1)^2 \times
\Rb \ri {\bf C}$ satisfies
\begin{eqnarray*}
\|h\|_{Y}^p :=  E \int_{\Rb} \int_0^1 \int_0^1
\frac{|h(t,s,x)|^p}{|t-s|}dsdt\; dx    < \infty.
\end{eqnarray*}
Let $\frac1p < \al_1<k < \al_2 < 1$. Denote
\begin{eqnarray}
X_{\nu}=L^p (  \Om \times (0,1), {\mathcal P}, {H}^{\al_\nu-1}_p(\Rb
) ),\quad Y_{\nu}=Y,\quad \nu = 1,2\nonumber
\end{eqnarray}
and define the operators $T_{i}:  X_{\nu}\ri Y_{\nu} $
($i=-1,-2,\ldots$) by
\begin{eqnarray*}
T_{i} \tilde g (\omega,t,s,x) =\left\{ \begin{array}{ll}
          \frac{S\tilde g  (\omega,t,x) - S\tilde g  (\omega,s,x)}{|t-s|^{\frac12k}},
            & \mbox{if         }  4^i \leq |t-s| < 4^{i+1}, \\
         0 ,& \mbox{otherwise}.
\end{array}
\right.
\end{eqnarray*}
Then, using Lemma \ref{20110324-3}, we have
\begin{eqnarray*}
\|T_{i}\tilde g \|_{Y_\nu} \lesssim 2^{i(\al_\nu -k)} \|\tilde g
\|_{X_{\nu}}, \quad \nu = 1,2,\quad i=-1,-2,\ldots.
\end{eqnarray*}
 As we take $\te = \frac{k - \al_1}{\al_2 - \al_1}$
and $\ga = 2^{\al_1 - \al_2}$, the norms $M_{i,\nu}$ of the map
$T_{i}:  X_{\nu}\ri Y_{\nu} $ satisfy
$$
M_{i\nu} \lesssim 2^{i(\al_\nu - k)} = c\ga^{i(\te - \nu)}.
$$
Note that $Y_{\te \infty} = Y $. Hence, by Lemma \ref{lemm4} we have
\begin{align} \label{R4}
E \int_{{\Rb}} \int\int_{| t-s| <1} \frac{|S\tilde g (t,x)- S\tilde
g (s,x)|^p}{|t-s|^{1 + \frac{p}2  k}}
   dsdt\;dx \lesssim \|\tilde g  \|^p_{ X_{\theta 1} },
\end{align}
where
\begin{eqnarray}
X_{\theta 1} &:=& \left(  L^p(  \Om \times (0,1) ,{\mathcal P},
{H}^{\al_1-1}_p(\Rb ) ), L^p(  \Om \times (0,1), {\mathcal P},
{H}^{\al_2-1}_p(\Rb ) ) \right)_{\te,1}.\nonumber
\end{eqnarray}

3. Now, choose $k_1,k_2$ and set $\eta\in(0,1)$ so that
\begin{eqnarray*}
\frac1p < \al_1 <k_1 < k< k_2 < \al_2 < 1,\quad \quad k =
(1-\eta)k_1 + \eta k_2.
\end{eqnarray*}
Denote $\te_{\mu} = \frac{k_{\mu} - \al_1}{\al_2 - \al_1}$,
$\mu=1,2$. Then (\ref{R4}) holds for the quadruples
$(\al_1,k_1,\al_2,\te_1)$ and $(\al_1,k_2,\al_2,\te_2)$. By Theorem
3.11.5 in \cite{BL} and lemma \ref{interpolationlemma} we have
\begin{align}
( X_{\theta_1 1}, X_{\te_2 1})_{\eta, p }
    & =(L^p(\Om \times (0,1) , {\mathcal P}, H_{p}^{\al_0-1} ({\mathbb
R}^{n})) ,
 L^p(\Om \times (0,1) ,{\mathcal P},
H_{p}^{\al_1-1} ({\mathbb R}^{n})) )_{\te, p}\nonumber\\
& =L^p(\Om \times (0,1) ,{\mathcal P}, ( H_{p}^{\al_0-1} ({\mathbb
R}^{n}), H_{p}^{\al_1-1} ({\mathbb R}^{n})_{\te,p})\nonumber\\
& = L^p(\Om \times (0,1) ,{\mathcal P}, B_p^{k-1} ({\mathbb
R}^{n})).\nonumber
\end{align}
On the other hand define the weights on  $d\pi:=dPdtds\,dx$ by
\begin{eqnarray*}
w_{\mu}=w_{\mu}(\omega,t,s,x)=\frac{1}{|t-s|^{1+\frac p2 k_{\mu}}},
\mu=1,2,\quad w=w^{1-\eta}_1\,w^{\eta}_2.
\end{eqnarray*}
Then by Theorem 5.4.1 (Stein-Weiss interpolation theorem) in
\cite{BL} we have
\begin{eqnarray}
\left(L^p(\Omega\times (0,1)^2\times
\mathbb{R}^n,w_1d\pi),L^p(\Omega\times (0,1)^2\times
\mathbb{R}^n,w_2d\pi)\right)_{\eta,p}=L^p(\Omega\times (0,1)^2\times
\mathbb{R}^n,w\,d\pi).\nonumber
\end{eqnarray}
Hence, we receive \eqref{tong}. Lemma \ref{maintheo} (1) now
follows. $\Box$

Now, we prove Lemma \ref{20110324-3}. We need the followings.
Recall that $\mathcal S(\Rb)$ is dense in any
$\mathcal{B}^k_p(\mathbb R^n)$.
\begin{lemm}\label{multiplier1}
Let $ l<0$, $1<q<\infty$ and $g \in \mathcal S(\Rb)$. Then the
followings hold.
\begin{itemize}
\item[(1)] For  $t>0$,
 \begin{align*}
 \left(\int_{\Rb}\left|\int_{\Rb} \Gamma(t,x-y)g(y) dy\right|^q dx \right)^{1/q}
 \lesssim (1+t^{\frac{l}2}) \|g \|_{H^l_p(\Rb)}.
 \end{align*}

\item[(2)] For   $t, \, h>0$,
\begin{align*}
\left(\int_{\Rb} \left|\int_{\Rb}
\Big(\Gamma(t+h,x-y)-\Gamma(t,x-y)\Big) g(y) dy\right|^q
dx\right)^{1/q} \lesssim h( t^{-1} +t^{\frac{l}2-1}) \|g
\|_{H^l_p(\Rb)}.
\end{align*}
\end{itemize}
\end{lemm}
\begin{proof}
(1) Denote $\mathcal{F}(h)=\hat h$, the spatial Fourier transform of
$h$. We observe that
\[\mathcal{F}(\Gamma(t,\cdot)*g)(\xi) = (1+|\xi|^{-l}) e^{-t|\xi|^2}\cdot m(\xi) (1+|\xi|^2)^{\frac l2} \widehat{g}(\xi),\]
where $m(\xi) = \frac{(1+|\xi|^2)^{-l/2}}{1+|\xi|^{-l}}$. We note
that $m$ is an $L^q$-Fourier multiplier, i.e., the operator $T_m$
defined by $\widehat{T_m(f)}(\xi)=m(\xi)\hat{f}(\xi)$ is
$L^q$-bounded. On the other hand we set
\[\widehat{K^t}(\xi) = (1+|\xi|^{-l}) e^{-t|\xi|^2}.\]
Since $\| \mathcal{F}^{-1}(\widehat{\phi}(\sqrt{t}\xi))\|_1 =
\|\phi\|_1$, we obtain
\begin{align*}
\|K^t\|_1 \le \|\mathcal{F}^{-1}(e^{-t|\xi|^2})\|_1 + t^{\frac l2}
\|\mathcal{F}^{-1}((t|\xi|^2)^{-\frac l2} e^{-t|\xi|^2})\|_1
\lesssim (1+ t^{\frac l2}).
\end{align*}
%Here, we used  $k<0$.
We have
\[\Gamma(t,\cdot)*g = K^t*(T_m(I-\Delta)^{\frac l2}g).\]
By Young's inequality and the multiplier theorem, we conclude that
for $1<q<\infty$
\begin{align*}
\left(\int_{\Rb} \left|\int_{\Rb} \Gamma(t,x-y)g(y) dy\right|^q
dx\right)^{1/q} \lesssim (1+ t^{\frac l2}) \|(I-\Delta)^{\frac
l2}g\|_q=(1+ t^{\frac l2}) \|g\|_{H^l_p(\mathbb{R}^n)}.
\end{align*}
(2) We set
\[\mathcal{F}((\Gamma(t + h,\cdot) - \Ga(t, \cdot))*g)(\xi) =
(-h|\xi|^2)(1+|\xi|^{-l}) e^{-t|\xi|^2}\cdot
\frac{1-e^{-h|\xi|^2}}{h|\xi|^2}\cdot m(\xi) (1+|\xi|^2)^{\frac
l2} \widehat{g}(\xi),\]
where $m(\xi) = \frac{(1+|\xi|^2)^{-l/2}}{1+|\xi|^{-l}}$. Note
that $\frac{1-e^{-h|\xi|^2}}{h|\xi|^2}$ is the $L^q$-Fourier
multiplier and the norm is independent of $h$. Set
\[\widehat{K^{t,h}}(\xi) = (-h|\xi|^2)(1+|\xi|^{-l}) e^{-t|\xi|^2}.\]
Then we have $\|K^{t,h}\|_1 \lesssim h( t^{-1} +t^{\frac l2-1})$
and the rest is similar to the case (1).
\end{proof}

\begin{lemm}\label{20110519-1}
Let $\frac1p<k<1$. Fix $i=-1,-2,\ldots$ and  denote $D_i := \{
(s,t) \in (0,1) \times (0,1) \, | \, 4^{i} \leq t-s < 4^{i+1 }
\}$. Consider the following operators $T_1,T_2,T_3$ which map
function defined on $(0,1)$ to a function defined on $D_i$:
\begin{align*}
\hspace{3cm}&(T_1 f)(s,t) : = \int_s^t (t-r)^{k-1} f(r) dr,\\
&(T_2 f)(s,t) :
= \int_{(s-4^i)\vee 0}^s (s-r)^{k-1} f(r) dr, \hspace{1cm} (s,t)\in D_i\\
&(T_3 f)(s,t) : = \int_0^{(s-4^i)\vee 0} (s-r)^{k-3} f(r) dr
\end{align*}

; note that $T_2f$ and $T_3$, in fact, are independent of $t$.
Then for $1\leq q < \infty$ we have
\begin{align}
\| T_m f\|_{L^q (D_i)} \leq c_m\, 4^{i(k+\frac1q)} \| f \|_{L^q
(0,1)},\; m=1,2\quad;\quad\| T_3 f\|_{L^q (D_i)} \leq c_3\,
4^{i(k-2+\frac1q)} \| f \|_{L^q (0,1)},\label{20110519-2}
\end{align}
where $c_1,c_2,c_3$ are absolute constants.
\end{lemm}
\begin{proof}
1. For $q =1$ Fubini's theorem gives us
\begin{align*}
\| T_1 f\|_{L^1 (D_i)}&  = \int_{4^i}^1 \int_{
(t-4^{i+1})\vee0}^{t- 4^{i}}
\int_s^t(t-r)^{k-1} |f(r)| drdsdt\\
 & \leq  \int_0^1  |f(r)|  \left[\int_r^{r + 4^{i+1}}  (t-r)^{k-1} \left(\int_{t - 4^{i+1}}^r  ds\right)dt \right]dr\\
& \leq \frac{4^{k+1}}{k(k+1)}\cdot 4^{i(k+1)} \| f\|_{L^1(0,1)}.
\end{align*}
For $ q= \infty$ we have
\begin{align*}
\sup_{(s,t) \in D_i} |(T_1f)(s,t)| &\leq \| f\|_{L^\infty(0,1)}
\cdot\sup_{(s,t) \in D_i} \int_s^t (t-r)^{k-1} dr   \leq
\frac{4^k}{k}\cdot 4^{ik} \| f\|_{L^\infty(0,1)}.
\end{align*}
Then, by the real interpolation theorem
$(L_1,L_{\infty})_{\theta,q}=L_q$ with the relation
$\frac1q=\frac{\theta}{1}+\frac{1-\theta}{\infty}=\theta$, we get
\begin{align*}
\| T_1 f \|_{L^q(D_i)} & \leq c\, (4^{i(k+1)})^{\theta}( 4^{ i k
})^{1-\theta}\,\| f\|_{L^q(0,1)}  \leq c\, 4^{i (k + \frac1q)
}\,\| f\|_{L^q(0,1)}.
\end{align*}
hence, (\ref{20110519-2}) for $T_1$ holds .

2. For $q =1$ we have
\begin{align*}
\| T_2 f\|_{L^1 (D_i)}&  = \int_{4^i}^1 \int_{
(t-4^{i+1})\vee0}^{t- 4^{i}}
\int_{(s-4^i)\vee 0}^s(s-r)^{k-1} |f(r)| drdsdt\\
 & \leq  \int_0^1  |f(r)|  \left[\int_r^{r + 4^{i}}  (s-r)^{k-1} \left(\int_{s+4^i}^{s+4^{i+1}}  dt\right)ds \right]dr\\
& \leq \frac{3}{k}\cdot 4^{i(1 + k)} \| f\|_{L^1(0,1)}
\end{align*}
and for $ q= \infty$
\begin{align*}
\sup_{(s,t) \in D_i} |(T_2 f)(s,t)| \leq \| f\|_{L^\infty(0,1)}
\cdot\sup_{(s,t) \in D_i} \int_{(s-4^i)\vee0}^s (s-r)^{k-1} dr
\leq \frac{1}{k}\cdot 4^{ki}\; \| f\|_{L^\infty(0,1)}.
\end{align*}
By the real interpolation theorem, (\ref{20110519-2}) for $T_2$
holds.

3. For $T_3$ the proof is similar. Observe that
\begin{align*}
\| T_3 f\|_{L^1 (D_i)}&  = \int_{4^i}^1 \int_{
(t-4^{i+1})\vee0}^{t- 4^{i}}
\int_{(s-4^i)\vee 0}^s(s-r)^{k-3} |f(r)| drdsdt\\
 & \leq  \int_0^1  |f(r)|  \left[\int_{r + 4^{i}}^1  (s-r)^{k-3} \left(\int_{s+4^i}^{s+4^{i+1}}  dt\right)ds \right]dr\\
& \leq \frac{3}{2-k}\cdot4^{i(k-1)} \| f\|_{L^1(0,1)}.
\end{align*}
and
\begin{align*}
\sup_{(s,t) \in D_i} |(T_3f)(s,t)| \leq \| f\|_{L^\infty(0,1)}
\cdot\sup_{(s,t) \in D_i} \int_0^{(s-4^i)\vee0} (s-r)^{k-3} dr
\leq \frac{1}{2-k}\cdot 4^{i(k-2)}\; \| f\|_{L^\infty(0,1)}.
\end{align*}
By the real interpolation theorem, (\ref{20110519-2}) for $T_3$
holds.
\end{proof}

\hspace{-0.5cm}{\bf Proof of Lemma \ref{20110324-3}} \quad  1.
\;Fix $i=-1,-2,\ldots$. Since
\begin{align} \label{half}
 &E \int_{{\Rb}} \int \int_{ 4^i<| t-s| <4^{i+1} }\frac{|v_3 (t,x)-
v_3 (s,x)|^p}{|t-s|^{1 + \frac{p}2  k}}
   dsdtdx \nonumber\\
  & \hspace{10mm}\leq  2  E \int_{{\Rb}} \int_0^1 \int^{t- 4^i}_{t-4^{i+1}}  \frac{|v_3 (t,x)- v_3 (s,x)|^p}{(t-s)^{1 + \frac{p}2  k}}
   dsdtdx,
\end{align}
we assume $t>s$. Note that
\begin{eqnarray}\label{difference}
\begin{array}{ll}
v_3(t,x) - v_3(s,x) &= \int_s^t \int_{\Rb} \Ga(t-r,x-y) \tilde
g(r,y) dy\,d w_r\vspace{0.4cm}\\ & \quad + \int_0^s \int_{\Rb}
(\Ga(t-r,x-y) - \Ga(s-r, x-y)) \tilde g(r,y) dy\,dw_r.
\end{array}
\end{eqnarray}
The right-hand side of \eqref{half} is bounded by the sum of the
following quantities (up to a constant multiple):
\begin{align*}
I_1 & = E \int_{\Rb}\int_0^1 \int^{(t- 4^i)\vee
0}_{(t-4^{i+1})\vee 0} \frac{|\int_s^t
 \int_{\Rb} \Ga(t-r,x-y)\tilde g(r,y) dyd w_r|^p}{(t-s)^{1 + \frac{p}2 k}}\;dsdtdx
 ,\\
I_2 & = E \int_{\Rb}\int_0^1 \int^{(t- 4^i)\vee
0}_{(t-4^{i+1})\vee 0} \frac{|\int_0^s \int_{\Rb} (\Ga(t-r,x-y) -
\Ga(s-r, x-y)) \tilde g(r,y) dydw_r|^p}{(t-s)^{1 + \frac{p}2
k}}\;dsdtdx.
\end{align*}
 2. Recall that we assume $\frac1p <
k<1$ and $p\ge 2$.

\noindent\underline{ Estimation of  $I_1$}\quad By
Burkholder-Davis-Gundy inequality(BDG) (see Section 2.7 in
\cite{K1}) $I_1$ is dominated  by, up to a constant multiple,
\begin{align}\label{main2}
&E \int_{\Rb} \int_{4^i}^1 \int_{(t-4^{i +1})\vee 0 }^{t- 4^{i}}
\frac{(\int_s^t  |\int_{\Rb}
 \Ga(t-r, x-y) \tilde g(r,y)dy|^2 dr)^{\frac{p}2}}{(t-s)^{1 + \frac{p}2
 k}}\;dsdtdx.
\end{align}
Next, by Minkowski's inequality for integrals  and Lemma
\ref{multiplier1} (1), the expression \eqref{main2} is bounded by,
up to a constant multiple,
\begin{align*}
&  E \int_{4^i}^1 \int_{(t-4^{i +1})\vee 0 }^{t- 4^{i}} \frac{
(\int_s^t ( \int_{\Rb} |\int_{\Rb}
 \Ga(t-r, x-y) \tilde g(r,y)dy|^pdx )^\frac2p dr)^\frac{p}2 }{(t-s)^{1 + \frac{p}2k}}\,dsdt\\
  \lesssim & \quad E \int_{4^i}^1 \int_{(t-4^{i +1})\vee 0 }^{t- 4^{i}}
 \frac{(\int_s^t(t-r)^{k-1} \| \tilde g(r,\cdot)\|^2_{H^{k-1}_p (\Rb)}dr)^\frac{p}2  }
 {(t-s)^{1 + \frac{p}2 k}}\,dsdt\\
  \lesssim & \quad 4^{-i(  1 + \frac{p}2k  )} E \int_{4^i}^1 \int_{(t-4^{i +1})\vee 0 }^{t- 4^{i}}
 \left(\int_s^t(t-r)^{k-1} \| \tilde g(r,\cdot)\|^2_{H^{k-1}_p (\Rb)}dr\right)^\frac{p}2  \,dsdt.
 \end{align*}
Applying Lemma \ref{20110519-1} with the operator $T_1$ and
$\frac{p}{2}$ in place of $q$, we receive
\begin{eqnarray}
I_1 \lesssim c\, \|\tilde g\|^p_{L^p ( \Om \times (0,1), \mathcal
P, H^{k-1}_p (\Rb))}.\nonumber
\end{eqnarray}

\noindent\underline{ Estimation of $I_2$}\quad BDG inequality
$I_2$ is dominated by, up to a constant multiple,
\begin{align*}
&E \int_{\Rb} \int_{4^i}^1 \int_{(t-4^{i +1})\vee 0 }^{t- 4^{i}}
\frac{(\int_0^s |\int_{\Rb}
 (\Ga(t-r,x-y) - \Ga(s-r,x-y)) \tilde g(r,y)dy|^2 dr)^{\frac{p}2}}{(t-s)^{1 +
 \frac{p}2k}}\,dsdtdx\\ \vspace{2mm}
\lesssim & \quad E \int_{\Rb} \int_{4^i}^1 \int_{(t-4^{i +1})\vee
0 }^{t- 4^{i}} \frac{(\int^s_{(s - 4^{i})\vee 0} |\int_{\Rb}
 (\Ga(t-r,x-y) - \Ga(s-r,x-y))\tilde  g(r,y)dy|^2 dr)^{\frac{p}2}}{(t-s)^{1 + \frac{p}2k}}\,dsdtdx\\
&\; + E \int_{\Rb} \int_{4^{i}}^1 \int_{(t-4^{i +1})\vee 0 }^{t-
4^{i}} \frac{(\int_0^{(s - 4^{i})\vee 0}  |\int_{\Rb}
 (\Ga(t-r,x-y) - \Ga(s-r,x-y))\tilde g(r,y)dy|^2 dr)^{\frac{p}2}}{(t-s)^{1 +
 \frac{p}2k}}\,dsdtdx\\
 =&\quad I_{21}+I_{22}.
\end{align*}
By Minkowski's inequality for integrals and  Lemma
\ref{multiplier1} (1) the term $I_{21}$ is bounded by, up to a
constant multiple,
\begin{align*}
&E  \int_{4^i}^1 \int_{(t-4^{i+1})\vee 0 }^{t- 4^{i}}
\frac{(\int^s_{(s - 4^{i})\vee 0}\; ((t-r)^{k-1} + (s-r)^{k-1} )
\|\tilde g(r,\cdot)\|^{2}_{H^{k-1}_p ({\mathbb R}^n)}
dr)^\frac{p}2 }{(t-s)^{1 +
\frac{p}2k}}\;dsdt\\
 \lesssim &\quad 4^{-i(1 + \frac{p}2k)} E \int_{4^i}^1
\int_{(t-4^{i+1})\vee 0 }^{t- 4^{i}}
 \left(\int^s_{(s - 4^{i})\vee 0}   (s-r)^{k-1}   \|
 \tilde g(r,\cdot)\|^{2}_{H^{k-1}_p({\mathbb R}^n)} dr\right)^\frac{p}2
 dsdt
\end{align*}
; we used $k<1$.  Lemma \ref{20110519-1} with the operator $T_1$
gives us
\begin{eqnarray}
I_{21}\lesssim c\, \|\tilde g\|^p_{L^p ( \Om \times (0,1),
\mathcal P, H^{k-1}_p (\Rb))}.\nonumber
\end{eqnarray}
By  Minkowski's inequality for integrals again  and Lemma
\ref{multiplier1} (2) the term $I_{22}$ is dominated by, up to a
constant multiple,
\begin{align*}
& E   \int_{4^{i}}^1 \int_{(t-4^{i+1})\vee 0 }^{t- 4^{i}}
\frac{(\int_0^{(s-4^{i})\vee0} (\int_{\Rb} |\int_{\Rb}
 (\Ga(t-r,x-y) - \Ga(s-r,x-y))\tilde g(r,y)dy|^p dx)^\frac2p dr)^{\frac{p}2}}{(t-s)^{1 + \frac{p}2k}}\;dsdt\\
&  \lesssim  4^{-i(1 + \frac{p}2k)} E\int_{ 4^{i}}^1
\int_{(t-4^{i+1})\vee 0 }^{t- 4^{i}}
\left(\int_0^{(s-4^{i})\vee0}\;
         (t-s)^2 (s-r)^{k-3}  \|\tilde g(r,\cdot) \|^{2}_{H^{k-1}_p(\Rb)} dr\right)^{\frac{p}2}\;dsdt\\
         &  \lesssim  4^{-i(1 + \frac{p}2k)}\cdot 4^{ip}\cdot E\int_{ 4^{i}}^1
\int_{(t-4^{i+1})\vee 0 }^{t- 4^{i}}
\left(\int_0^{(s-4^{i})\vee0}\;
         (s-r)^{k-3}  \|\tilde g(r,\cdot) \|^{2}_{H^{k-1}_p(\Rb)}
         dr\right)^{\frac{p}2}\;dsdt.
\end{align*}
Then Lemma \ref{20110519-1} with the operator $T_3$ gives us
\begin{eqnarray}
I_{22}\lesssim c\, \|\tilde g\|^p_{L^p ( \Om \times (0,1),
\mathcal P, H^{k-1}_p (\Rb))}.\nonumber
\end{eqnarray}

3. By the estimations of $I_1,I_2$ our claim (\ref{main3})
follows. \hfill $\Box$

\mysection{Proof of Lemma \ref{maintheo} (2)}\label{20101216-2}
Again, we just assume $T=1$. We start with the following lemmas.
\begin{lemm}\label{Ga}
For $0 < t,\; r < \infty$
\begin{align*}
\int_{\Rb} | \Ga(t+ r,y) - \Ga(r,y)| dy \lesssim
\left\{\begin{array}{l}
 \frac{t}r, \quad t < r, \\ \\
 1, \quad t \geq r.
\end{array}
\right.
\end{align*}

\end{lemm}
; this is almost obvious and the proof is omitted.

\begin{lemm}\label{lemma-complex}
Let $0\le  \te < 1, \,\, 1 < p< \infty$. Then for $g \in
H^\te_{p,o}(D)$,
\begin{align*}%\label{complex-interpolation}
\int_D \de(y)^{-p\te} |g(y)|^p dy \leq c \| g\|_{H^\te_{p,o}(D)}^p,
\end{align*}
where $\de(y) = dist(y, \pa D)$. The constant $c$ depends only on
$p, \, n$.
\end{lemm}
\begin{proof}
We may assume $0<\theta<1$. We use complex interpolation of
$L^p$-spaces of measures. Let $d\mu_0(y) = dy$ and   $d\mu_1(y) =
\de^{-p}(y) dy$. The complex interpolation space between $L^p (d
\mu_0)$ and $L^p(d\mu_1)$ with index $\te$ is
\begin{eqnarray}
(L^p(d\mu_0), L^p(d\mu_1))_{[\te]} = L^p(d\mu_\te),\quad d\mu_\te
(y) := \de^{-p\te} dy\nonumber
\end{eqnarray}
(see Theorem 5.5.3 in \cite{BL}). Note that using Hardy's
inequality, we obtain  that for  $g \in H^1_{p,o}(D)$
\begin{align*}
\left(\int_D \de(y)^{-p} |g(y)|^p dy\right)^{\frac1p} \leq c
\left(\int_D |\na g(y)|^p dy\right)^{\frac1p} = c\; \|
g\|_{H^1_{p,o}(D)}.
\end{align*}
Since $( H^1_{p,o} (D), L^p(D))_{[\te]}= H^\te_{p,o}(D)$ (see Proposition 2.1 in  \cite{JK}), we get
\begin{align*}
\left(\int_D \de^{-p\te} (y) |g(y)|^p dy\right)^{\frac1p} \leq c\;
\| g\|_{( H^1_{p,o} (D), L^p(D))_{[\te]}} = c\; \|
g\|_{H^\te_{p,o}(D)}.
\end{align*}

\end{proof}

{\bf Proof of Lemma \ref{maintheo} (2)}\quad  1. Recall $1\le k
<1+\frac1p$ and $p\ge 2$.  For $g \in {\mathbb
H}^{k-1}_p(D_T)=L^p(\Om\times(0,1),{\mathcal P},
{H}^{k-1}_{p,o}(D))$, we denote $\tilde g \in
L^p(\Om\times(0,1),{\mathcal P}, {H}^{k-1}_p(\Rb))$ by $\tilde
g(\om,t,x) = g(\om,t,x)$ for $x \in D$ and $\tilde g(\om,t,x) =0$
for $x \in \Rb\setminus \bar D$. Then  by lemma \ref{prop3}, we
have $v_3 \in L^p(\Om\times (0,1), \mathcal P, H^{k}_p (\Rb))$,
where
\begin{eqnarray}
v_3(t,x)=\int_0^t \int_{\Rb} \Ga(t-s, x-y) \tilde g(s,y)
dy\;dw_s.\nonumber
\end{eqnarray}
By the usual trace theorem (see \cite{JW}), we get $v_3|_{\pa D_T}
\in L^p(\Om \times (0,1), {\mathcal P}, {B}^{k -\frac1p }_p (\pa
D))$. Hence, it is sufficient   to show that
\begin{align}\label{Main2}
E \int_{\pa D} \int \int_{0<s<t<1} \frac{|v_3(x,t) -
v_3(x,s)|^p}{(t-s)^{1 + \frac{p}2(k -\frac1p)}}dsdt\;d\si(x)
\lesssim \| g\|_{L^p ((0,1), \mathcal P, H^{k-1}_p (D))}.
\end{align}
Then, using real interpolation (see lemma  \ref{interpolationlemma}),
we complete the proof of lemma \ref{maintheo} (2).

2. The left-hand side of \eqref{Main2} is bounded by the sum of
the following quantities (up to a constant multiple):
\begin{align*}
J_1 & = E \int_{\pa D} \int_0^1 \int_0^t
\frac{|\int_s^t \int_{D} \Ga(t-r,x-y) \tilde g(r,y) dyd w_r|^p}{(t-s)^{1 + \frac{p}2(k -\frac1p)}}dsdtd\sigma(x) ,\\
J_2 & = E \int_{\pa D} \int_0^1 \int_0^t   \frac{|\int_0^s
\int_{D} (\Ga(t-r,x-y) - \Ga(s-r, x-y)) \tilde g(r,y)
dydw_r|^p}{(t-s)^{1 + \frac{p}2(k -\frac1p)}}dsdtd\si(x).
\end{align*}

\noindent\underline{Estimation  of $J_1$}\quad By BDG's
inequality, $J_1$ is dominated by, up to a constant multiple,
\begin{align}\label{Main1}
E \int_{\pa D} \int_0^1 \int_0^t \frac{(\int_s^t  |\int_{D}
 \Ga(t-r, x-y) \tilde g(r,y)dy|^2 dr)^{\frac{p}2}}{(t-s)^{1 + \frac{p}2(k
 -\frac1p)}}dsdtd\si(x).
\end{align}
Note
\begin{align}\label{J-1}
\begin{array}{ll}
\vspace{2mm} &\int_{\pa D} \Big(\int_s^t  |\int_{D}
 \Ga(t-r, x-y)\tilde  g(r,y)dy|^2 dr \Big)^{\frac{p}2}d\si(x)\\ \vspace{2mm}
 \lesssim &\Big(\int_s^t \Big(\int_{\pa D}  |\int_{D}
 \Ga(t-r, x-y)  \tilde g(r,y)dy|^p d\si(x) \Big)^{\frac{2}p}dr \Big)^{\frac{p}2}
 \\ \vspace{2mm}
 \lesssim &\Big(\int_s^t \left(\int_{\pa D} (\int_{\mathbb{R}^n}
 \Ga(t-r, x-y)dy)^{\frac{p}{p'}}  \cdot \int_{D}
 \Ga(t-r, x-y) |\tilde g(r,y)|^p dy\; d\si(x) \right)^{\frac{2}p}dr \Big)^{\frac{p}2}\\
  \lesssim &\Big(\int_s^t  \Big( \int_{D} |\tilde g(r,y)|^p \;\int_{\pa D}
 \Ga(t-r, x-y)d\si(x) \;dy \Big)^{\frac{2}p}dr
 \Big)^{\frac{p}2},\vspace{2mm}
 \end{array}
\end{align}
where $\frac1p + \frac1{p'} =1$. Note that for $y \in D$ there  is
a $x_y \in \pa D$ such that $\de(y) = |y -x_y|$, where $\de(y) =
dist(y, \pa D)$. Since $D$ is a bounded Lipschitz domain, there is
$r_0> 0$ independent of $x_y$ such that
 $|y -x| \approx \de(y) + |x -x_y|$ for all $ | x - x_y| < r_0$. We have
 \vspace{2mm}
\begin{align}\label{integral}
\begin{array}{ll}
\vspace{2mm} &\int_{\pa D}  \Ga(t-r, x-y)d\si(x) \\ \vspace{2mm}
\lesssim & \int_{|x - x_y| < r_0} \left[(t-r)^{-\frac{n}2}\cdot
e^{-c\frac{ \de(y)^2 +|x-x_y|^2}{t-r}}\right]d\si(x) + \int_{|x-
x_y| \geq r_0} \left[(t-r)^{-\frac{n}2} \cdot e^{-c\frac{ \de(y)^2
+|x-x_y|^2}{t-r}}\right]d\si(x)
\\ \vspace{2mm}
 \lesssim &\int_{|x'| < r_0, \, x' \in {\mathbb R}^{n-1} }
\left[(t-r)^{-\frac{n}2} \cdot e^{-c\frac{ \de(y)^2
+|x'|^2}{t-r}}\right]d x' +  (t-r)^{-\frac{n}2}\cdot  e^{-c\frac{
\de(y)^2 +r_0^2}{t-r}} \\ \vspace{2mm}
  \lesssim & (t-r)^{-\frac12} \cdot e^{- c\frac{\de(y)^2}{t-r}}
  \left[\int_{\mathbb{R}^{n-1}}e^{-c|y'|^2}dy'+(t-r)^{\frac{n-1}2}\cdot
  e^{-c\frac{r_0^2}{t-r}}\right]\\
 \lesssim & (t-r)^{-\frac12} \cdot e^{- c\frac{\de(y)^2}{t-r}}.
 \end{array}
\end{align} By
\eqref{integral} and the H\"older inequality, the last term in
\eqref{J-1} is bounded by, up to a constant multiple,
\begin{align} \label{J-2} \vspace{2mm}
  \Big(\int_s^t  \Big( \int_{D} |\tilde g(r,y)|^p (t-r)^{-\frac12}
e^{-c \frac{\de(y)^2}{t-r}} dy \Big)^{\frac{2}p}dr
\Big)^{\frac{p}2}
 & \lesssim (t-s)^{\frac{p-2}2} \int_s^t   \int_{D} |\tilde g(r,y)|^p
(t-r)^{-\frac12} e^{-c\frac{\de(y)^2}{t-r}} dy dr.
\end{align}
Hence, via Fubini's Theorem,  \eqref{Main1} is dominated by, up to
a constant multiple,
\begin{align*}
&E \int_0^1  \int_{D} |\tilde g(r,y)|^p  \left[\int_r^1 \int_0^r
(t-s)^{-\frac{p}2 (k-1) -\frac32} \frac{1}{(t-r)^\frac12}
e^{-c\frac{\de(y)^2}{t-r}}dsdt\right]dydr\\
 & \lesssim  E \int_0^1  \int_{D} |\tilde g(r,y)|^p  \left[\int_r^1
e^{-c\frac{\de(y)^2}{t-r}}\cdot
(t-r)^{-\frac{p}2 (k-1) -1 }  dt\right]dydr\\
 & = E \int_0^1  \int_{D} |\tilde g(r,y)|^p \left[\int_0^{1-r}
e^{-c\frac{\de(y)^2}{t}} \cdot t^{-\frac{p}2 (k-1) -1} dt \right]dydr \\
&= E \int_0^1 \int_{D} \de^{-p(k-1)}(y) |\tilde g(r,y)|^p
\left[\int^\infty_{\frac{\de^2(y)}{1-r}} e^{- ct}\cdot
t^{\frac{p}2 (k-1)
-1}dt\right]dydr\\
&\lesssim E \int_0^1 \int_{D} \de^{-p(k-1)}(y)
 |\tilde g(r,y)|^p dydr\\
 & \lesssim E \int_0^1 \|g(\cdot, r)\|_{H_{p,o}^{k-1} (D)}^p dr
\end{align*}
; for the last inequality we used the assumption $g \in {\mathbb
H}^{k-1}_{p,o} (D)$ and Lemma \ref{lemma-complex} with
$\theta=k-1$.

\noindent \underline{ Estimation of  $J_2$}\quad By BDG's
inequality, $J_2$ is dominated by, up to a constant multiple,
\begin{align}\label{main4}
E \int_{\pa D} \int_0^1 \int_0^t \frac{(\int_0^s  |\int_{D}
 (\Ga(t-r, x-y) - \Ga(s-r,x-y)) \tilde g(r,y)dy|^2 dr)^{\frac{p}2}}{(t-s)^{1 + \frac{p}2(k
 -\frac1p)}}dsdtd\si(x).
\end{align}
Define $A:= A(t,s,r,x,y) = \Ga(t-r, x-y) - \Ga(s-r,x -y)$. If
$p>2$, using the H\"older inequality twice, we get
\begin{align}
  \Big(\int_0^s  \left|\int_{D}
 A\cdot \tilde g(r,y)dy\right|^2 dr\Big)^{\frac{p}2}
 & \leq  \Big(\int_0^s \left[\int_{D} |A|dy \right]^{\frac{2(p-1)}p}
 \left[\int_{D} |A| |\tilde g(r,y)|^pdy \right]^{\frac{2}p}
 dr\Big)^{\frac{p}2}\label{20110519-4}
  \\
 & \leq  \Big(\int_0^s \left[\int_{D}
 |A|dy \right]^{\frac{2(p-1)}{p-2}}dr\Big)^{\frac{p-2}{2}}
 \int_0^s \int_{D} |A| |\tilde g(r,y)|^p dydr.\nonumber
\end{align}
Next, by changing variable from $r$ to $s-r$ and Lemma \ref{Ga},
\begin{align*}
\int_0^s \Big(\int_{D}
 |A|dy \Big)^{\frac{2(p-1)}{p-2}}dr &  = \int_0^s \Big(\int_{D}
|\Ga(t-s + r, x -y) - \Ga(r, x -y)|dy \Big)^{\frac{2(p-1)}{p-2}} dr \\
& \lesssim \left\{\begin{array}{l}
      \int_0^s dr, \quad s < t-s \\
         \int_0^{t-s} dr + \int_{t-s}^s (\frac{t-s}{r} )^{\frac{2(p-1)}{p-2}}
         dr, \quad s \geq  t-s
 \end{array} \right.  \\
 & = \left\{\begin{array}{l}
      s, \quad s < t-s \\
         t-s + (t-s)^{\frac{2(p-1)}{p-2}} ( (t-s)^{-\frac{p}{p-2}} - s^{-\frac{p}{p-2}}) \quad s \geq t-s
 \end{array} \right.  \\
 &  \lesssim (t-s)
\end{align*}
and
\begin{align}
\Big(\int_0^s  \left|\int_{D}
 A\cdot \tilde g(r,y)dy\right|^2 dr\Big)^{\frac{p}2}\lesssim
 (t-s)^{\frac{p-2}{2}}\int_0^s \int_{D} |A| |\tilde g(r,y)|^p dydr.\label{20110519-3}
\end{align}
If $p=2$, (\ref{20110519-4}) with $p=2$ and Lemma \ref{Ga}
immediately yields (\ref{20110519-3}). Hence, \eqref{main4} is
dominated by, up to a constant multiple,
\begin{eqnarray}
&& E \int_{\pa D} \int_0^1 \int_0^t \left[\int_0^s \int_{D}
|A(t,s,r,x,y)|
|\tilde g(r,y)|^pdydr\right]  (t-s)^{-\frac{p}2 (k-1) -\frac32} dsdt\;d\si(x)\nonumber\\
 &\lesssim&  E \int_0^1 \int_{D} |\tilde g(r,y)|^p \left[\int_0^{1-r} \int_0^t
(t-s)^{-\frac{p}2 (k-1) -\frac32} \int_{\pa D}|\Ga(t,x -y) -
\Ga(s, x-y)| d\si(x)\;dsdt\right]dydr.\nonumber\\ \label{main}
\end{eqnarray}
We estimate the boundary ($\partial D$) integral part: Since $ s<
t$, we have
\begin{align*}
& \int_{\pa D}|\Ga(t,x -y) - \Ga(s,x-y)| d\si(x)\\
&  \leq
 (\frac{1}{s^{\frac12 n}} - \frac{1}{t^{\frac12n}})\int_{\pa D}
 e^{-\frac{|x -y|^2}{4s}} d\si(x) + t^{-\frac12 n} \int_{\pa D}  e^{-\frac{|x
 -y|^2}{4t}}- e^{-\frac{|x -y|^2}{4s}}d\si(x) \\
 & = K_1 + K_2.
\end{align*}
Applying \eqref{integral} again,
\begin{align*}
K_1 = \frac{t^{\frac12 n} - s^{\frac12 n}}{t^{\frac12 n}
s^{\frac12 n}} \int_{\pa D} e^{-\frac{|x -y|^2}{4s}} d\si(x) &
\leq \frac{t^{\frac12 n} - s^{\frac12 n}}{t^{\frac12 n} }
s^{-\frac12} e^{-c \frac{\de(y)^2}{s}}\\
& \leq \left\{\begin{array}{l}
    s^{-\frac12} e^{- c\frac{\de(y)^2}{s}}  ,\quad   0< s < \frac12 t,\\
     t^{-\frac32}(t-s) e^{-c \frac{\de(y)^2}{t}}  , \quad \frac12 t \leq s< t.
\end{array}
\right.
\end{align*}
For $K_2$ we consider two cases. If $0 <s < \frac12 t$, using
\eqref{integral}, we get
\begin{align*}
K_2 \leq t^{-\frac12 n} \int_{\pa D} e^{-\frac{|x -y|^2}{4t}}
d\si(x) \leq  t^{ -\frac12} e^{- c\frac{\de(y)^2}{t}}.
\end{align*}
For $ \frac12 t < s < t $, using the Mean Value Theorem, there is
a $\eta$ satisfying $s < \eta < t$ such that
\begin{eqnarray}
K_2 = t^{-\frac12 n} \int_{\pa D} (t-s) \frac{|x-y|^2}{4 \eta^2}
e^{-\frac{|x -y|^2}{\eta}} d\si(x)\nonumber
\end{eqnarray}
and this leads to
\begin{align*}
 K_2 & \lesssim  t^{-\frac12 n} \int_{\pa D} (t-s) \frac{|x-y|^2}{t^2}
e^{-\frac{|x -y|^2}{4 t}} d\si(x)\\
& \lesssim t^{-\frac12 n -2}(t-s)  \int_{\pa D}(|x - x_y|^2 +\de(y)^2)  e^{-c\frac{|x-x_y|^2 +\de(y)^2}{t}}d\si(x)\\
 & \lesssim  t^{-\frac12 n -2}(t-s) e^{-c\frac{\de(y)^2}{t}}(t^{\frac{n+1}2 } + \de^2(y)t^{\frac{n-1}2})\\
& =   (t-s) e^{-c \frac{\de(y)^2}{t}}(t^{-\frac{3}{2}} + \de(y)^2
t^{-\frac{5 }2}).
\end{align*}
By these estimations, the bracket in (\ref{main}) is bounded by,
up to a constant multiple,
\begin{align*}
&\; \int_0^{1-r} t^{-\frac{p}2  (k-1) -\frac32}
 \left[\int_0^{\frac12 t} ( t^{ -\frac12} e^{-c\frac{\de(y)^2}{t}} + s^{
-\frac12} e^{-c\frac{\de(y)^2}{s}})ds\right]dt\\
+ & \int_0^{1-r} e^{-c\frac{\de(y)^2}{t}} (t^{-\frac{3}{2}} +
\de(y)^2t^{-\frac{5 }2} ) \left[\int_{\frac12 t}^t
(t-s)^{-\frac{p}2 (k-1) -\frac12}      ds\right]dt \\
 \lesssim & \int_0^{1-r} \left[ t^{-\frac{p}2  (k-1) -1}  e^{-c\frac{\de(y)^2}{t}}
+ t^{-\frac{p}2  (k-1) -\frac32}\cdot \de(y) \cdot
\int^\infty_{\frac{2\de(y)^2}{t}} s^{-\frac32} e^{-cs}ds   +
t^{-\frac{p}2 (k-1) -2} \cdot e^{-c\frac{\de(y)^2}{t}} \cdot
\de(y)^2 \right]dt\\
&=:L_1+L_2+L_3
\end{align*}
; for the inequality we used the assumption $k< 1+\frac1p$. It is
easy to see that the terms $L_1$ and $L_3$ are dominated by
$\de(y)^{-p(k-1)}$. This is also true for $L_2$; if $2\de(y)^2
\geq (1-r)$, then $2\de(y)^2 \geq t$ and
\begin{align*}
 L_2\lesssim \de(y)  \int_0^{1-r} t^{-\frac{p}2  (k-1) -\frac32} \int^\infty_{\frac{2\de(y)^2}{t}} s^{-\frac32} e^{-cs} dsdt
 &\lesssim \de(y)  \int_0^{1-r}   t^{-\frac{p}2  (k-1) -\frac32} e^{-c\frac{\de(y)^2}{t}} dt \lesssim \de(y)^{-p(k-1)}.
\end{align*}
If $2\de(y)^2 \leq 1-r$,
\begin{align*}
L_2 & \lesssim \de(y)  \int_0^{1-r} t^{-\frac{p}2  (k-1) -\frac32}
\int^\infty_{\frac{2\de(y)^2}{t}} s^{-\frac32} e^{-cs} dsdt\\
 &
\lesssim \de(y)  \int_0^{2\de(y)^2} t^{-\frac{p}2  (k-1) -\frac32}
e^{-c\frac{\de(y)^2}{t}}dt +
\int_{2\de(y)^2}^{1-r} t^{-\frac{p}2  (k-1) -1} dt\\
&\lesssim \de(y)^{-p(k-1)}.
\end{align*}
After all, (\ref{main}) (hence $J_2$) is bounded by, up to a
constant multiple,
\begin{align*}
 E \int_0^1 \int_{D}\de(y)^{-p(k-1)} |\tilde g(r,y)| dydr
 & \lesssim E \int_0^1 \|g(\cdot, r)\|_{H_p^{k-1} (D)}^p dr
\end{align*}
; we used the assumption $g \in {\mathbb H}^{k-1}_{p,o} (D)$ and
Lemma \ref{lemma-complex}.

3. The step 2 implies \eqref{Main2}. The lemma is proved.
\hspace{7cm}$\Box$

\end{document}